\documentclass[11pt]{article}
\usepackage{epsfig}
\usepackage{amssymb,amsmath,amsthm,amscd}
\usepackage{latexsym}

\newtheorem{thm}{Theorem}[section]
\newtheorem{prop}[thm]{Proposition}

\newtheorem{lem}[thm]{Lemma}

\theoremstyle{definition}

\theoremstyle{remark}
\newtheorem{rem}{Remark}[section]
\def\eqref#1{(\ref{#1})}

\newcommand{\bbR}{{\mathbb{R}}}

\newcommand{\bbD}{{\mathbb{D}}}

\newcommand{\Lap}{\triangle}

\newcounter{labelflag} 

\newcommand{\Label}[1]{
                       \ifnum\thelabelflag=1
                          \ifmmode
                             \makebox[0in][l]{\qquad\fbox{\rm#1}}
                          \else
                             \marginpar{\vspace{0.7\baselineskip}
                                        \hspace{-1.1\textwidth}
                                        \fbox{\rm#1}}
                          \fi
                       \fi
                       \label{#1}
                      }

\begin{document}

\baselineskip =1.2\baselineskip

\begin{titlepage}
\title{Poisson-Nernst-Planck Systems for Narrow Tubular-like Membrane Channels}
\author{Weishi Liu\footnote{Partially supported by NSF Grant DMS-0406998.}\\
 Department of Mathematics\\ University of Kansas,
Lawrence, KS 66045 \vspace{1cm}\\
Bixiang Wang\footnote{Partially supported by NSF Grant DMS-0703521.}\\
Department of Mathematics\\
New Mexico Institute of Mining and Technology, Socorro, NM 87801}
\date{}
\end{titlepage}

\maketitle

\begin{abstract}  We study global dynamics of the Poisson-Nernst-Planck (PNP)
system for flows of two types of ions through a narrow tubular-like membrane channel.
As the    radius of the cross-section of the three-dimensional tubular-like membrane
channel approaches zero, a one-dimensional limiting PNP system is derived.
This one-dimensional limiting system differs from previous studied one-dimensional
 PNP systems in that it encodes the defining geometry of the three-dimensional membrane
channel. To justify  this limiting process, we show that the global  attractors
of the three-dimensional PNP systems  are upper semi-continuous to that
 of the  limiting PNP system. We then examine the dynamics of the
one-dimensional limiting PNP system. For large Debye number, the
steady-state of the one-dimensional limiting PNP system is completed
 analyzed using the geometric singular perturbation theory. For a special
case, an entropy-type Lyapunov functional is constructed to show the global,
asymptotic stability of the steady-state.
\end{abstract}



\newpage

\section{Introduction}
\setcounter{equation}{0}
 Poisson-Nernst-Planck  (PNP) systems serve as basic electro-diffusion equations modeling, for example, ion flow through   membrane channels, transport of holes and electrons in  semiconductors  (see, e.g., \cite{BCE, BCEJ, KS, Rub} and the references therein).
There are many excellent works   on derivation of PNP systems from Boltzmann equations (see \cite{BCE} and reference therein) by assuming the collision time is much smaller than the characteristic time.
The PNP systems have been studied under various physical relevant boundary conditions
such as  the non-flux, homogeneous Dirichlet and  Neumann boundary conditions.
 For those types of boundary conditions, in addition to the total charge conservation   and the existence of various first integrals, Boltzmann H-functionals or entropy-like functionals are successfully constructed, which, together with the advances of Csisz\'ar-Kullback-type or logarithmic Soblev inequalities, are applied to investigate the  asymptotic behavior of the PNP systems and stability  of
 steady-state or self-similar solutions (see, e.g., \cite{BD, BHN1, GG1, GG2, Gro1,  Moc1, Seid1}). In the context of ion flow through membrane channels, it is physically unreasonable to impose the above mentioned boundary conditions on the whole boundary, particularly at the two ``ends'' of the channels. Instead, non-homogeneous Dirichlet conditions on the two ``ends'' are typically assumed. PNP systems supplemented with this type of boundary conditions result in quite different dynamical behavior.  The total charges within the channels are not conserved and entropy-like functionals are not available in general, and, most significantly, the asymptotic behavior is different as showed in special cases in the present paper.

In this  work, we  start a systematic  study   of the  PNP systems
modeling   ion flows through narrow tubular-like membrane channels
 in physiology.  For definiteness,  consider flow of two types of ions, $S_1$ and $S_2$, one with the positive valence $\alpha_1>0$ and the other with the negative valence $-\alpha_2<0$, passing through a membrane channel   with length
 normalized from $X=0$ to $X=1$.   Denote the concentrations of
$S_1$ and $S_2$ at location $(X,Y,Z)$ and   time $t$ by
$c_1(t, X,Y,Z)$ and $c_2(t, X,Y,Z)$.
Then the electric potential  $\Phi(t, X,Y,Z)$ in the channel   is
 governed  by the Poisson equation
\[  \Lap\Phi=-\lambda(\alpha_1 c_1-\alpha_2c_2),\]
where the parameter  $ \lambda$ is the Debye number related to the ratio of the Debye length to a characteristic  length scale. The  flux densities,
$\bar{J}_1$  and $\bar{J}_2$, of the two ions contributed from the
concentration gradients of the two ions and  the electric field satisfy the
Nernst-Planck equations
\[D_1(\nabla c_1+\alpha_1 c_1\nabla \Phi)=-\bar{J}_1\; \mbox{ and }\;
      D_2 (\nabla c_2-\alpha_2 c_2\nabla \Phi)=-\bar{J}_2,\]
      and
      the  continuity equations
 \[\frac{\partial c_1}{\partial t}+ \nabla\bar{J}_1 =0,
 \quad
 \frac{\partial c_2}{\partial t}+ \nabla\bar{J}_2=0,\]
  where $D_1$ and $D_2$ are the diffusion constants of ions $S_1$ and $S_2$
 relative to the membrane channel.
The Poisson-Nernst-Planck system is thus given by
\begin{align}\label{pnp}\begin{split}
  \Lap\Phi=&-\lambda (\alpha_1 c_1-\alpha_2c_2),\\
 \frac{\partial c_1}{\partial t}=& D_1\nabla \cdot
   (\nabla c_1+\alpha_1 c_1\nabla \Phi),\\
    \frac{\partial c_2}{\partial t}=& D_2 \nabla \cdot
 (\nabla c_2-\alpha_2 c_2\nabla \Phi).
 \end{split}
 \end{align}

  PNP systems have been studied by many authors
 (see, e.g., \cite{BCE, BCEJ, BD,     BHN1, GG1, GG2, Gro1, Hom,  Jer,  JeK,   KS, Moc1,       PJ,  Rub,  Seid1}).  Many   works have been attributed to the one-dimensional PNP systems and particularly the steady-state problems (see, e.g.,~\cite{Hom, Jer, PJ, JeK}). Consideration of one-dimensional PNP systems is motivated naturally by the fact that membrane channels are narrow. To make this reduction more rigorous, we present a mathematical framework based on the ideas in \cite{HR1, HR2} and investigate mathematically the limiting process as the three-dimensional domain shrinks to a line segment. More specifically,  starting with the   situation that $\Omega_{\epsilon}$ is a revolution domain about its length ($\epsilon$ is related to the maximal radius of the cross-sections of the channel), we will derive a one-dimensional limiting PNP system as  $\epsilon\to 0$. Differing from the simple one-dimensional version of the  PNP system, this limiting PNP system encodes the  defining geometry of the three-dimensional channel. As the first step in justifying the limiting process,   we show  the upper semi-continuity of the  global attractors ${\mathcal A}_{\epsilon}$
 of  the three-dimensional systems at $\epsilon=0$. The existence of global attractors for the PNP systems  can be found in \cite{GG2}. It is expected that if the one-dimensional limiting system is structurally stable, then its dynamics determines that of three-dimensional system for $\epsilon>0$ small. We will thus examine, in this paper, the steady-state problem of the one-dimensional limiting system. For large Debye number, the steady-state problem can be viewed as a singularly perturbed one. We show that this problem can be completely analyzed using the geometric singular perturbation theory as in~\cite{Liu1}.

 The rest of the paper is organized as follows. In Section 2, we give detailed formulation of our problem. The domain for the three-dimensional PNP system will be specified and, as the domain shrinks to a one-dimensional segment, a one-dimensional limiting PNP system is derived. We then state our results  on the upper semi-continuity of   attractors, on the singular boundary value problem of the steady-state PNP system. The proofs are provided in Sections 3 and 4.  In Section 3,   after some  technical preparations, we  show the   upper semi-continuity of   attractors. Section 4 is devoted to the geometric analysis of the singularly perturbed steady-state problem of the one-dimensional limiting PNP system. At the end of this section, a special case is studied for which a $H$-function is found and used to establish the asymptotic stability of the steady-state.

\section{Formulation of the problem and  the statements of main results}
\setcounter{equation}{0}
\subsection{Three-dimensional PNP and a one-dimensional limit}
 We  start  with setting up   our problem.  The membrane channel considered here is special and  will be  viewed as a tubular-like domain $\Omega_{\epsilon}$ in $  \bbR^3$ as follows:
\[\Omega_{\epsilon}=\{(X,Y,Z):  0<X<1, \; Y^2+Z^2  < g^2 (X,\epsilon)\},\]
where   $g$ is a  smooth (at least ${\mathcal C}^3$)    function  satisfying
 \begin{equation}\label{gC}
 g(X,0)=0\;\mbox{ and }\; g_0(X)=\frac{\partial g}{\partial\epsilon}(X,0)>0 \;\mbox{ for
}\; X \in [0,1].
\end{equation}
The positive parameter $\epsilon$  measures the sizes of cross-sections of the membrane channel.  For a technical reason (used in Lemma~\ref{foliation}), we also
assume that
\[ {\frac {\partial g}{\partial X}} (0,\epsilon)= {\frac {\partial g}{\partial X}}   (1,\epsilon)=0.\]
The boundary  $\partial\Omega_{\epsilon}$ of $\Omega_{\epsilon}$
will be     divided into three portions as follows:
\begin{align*}
\hat{L}_{\epsilon}=&\{(X,Y,Z)\in\partial\Omega_{\epsilon}: X=0\},\\
\hat{R}_{\epsilon}=&\{ (X,Y,Z)\in\partial\Omega_{\epsilon}: X=1\},\\
\hat{M}_{\epsilon}=&\{ (X,Y,Z)\in\partial\Omega_{\epsilon}:
   Y^2+Z^2= g^2(X, \epsilon)\}.
 \end{align*}
Thus, $\hat{L}_{\epsilon}$ and $\hat{R}_{\epsilon}$ are viewed as the two ends of the channel and $\hat{M}_{\epsilon}$ the wall of the channel. The boundary conditions considered in this paper are
\begin{align}\label{bv}\begin{split}
\Phi|_{\hat{L}_{\epsilon}} =\phi_0>0,   \quad  \Phi|_{\hat{R}_{\epsilon}}=0, &
 \quad
c_k|_{\hat{L}_{\epsilon }} =l_k>0,  \quad c_k|_{\hat{R}_{\epsilon }}=r_k>0,\\
\frac{\partial \Phi}{\partial {\bf n}}|_{\hat{M}_{\epsilon }}
 &=\frac{\partial c_k}{\partial {\bf n}}|_{\hat{M}_{\epsilon }}= 0,
 \end{split}
\end{align}
where $\phi_0$, $l_k$ and $r_k$ ($k=1,2$) are  constants, and  ${\bf n}$ is the outward unit  normal vector to $\hat{M}_{\epsilon }$.  Although   the most natural  boundary conditions on $\hat{M}_{\epsilon}$ would be the non-flux one
\[\left(\frac{\partial c_1}{\partial {\bf n}}+\alpha_1c_1\frac{\partial \Phi}
 {\partial {\bf n}}\right)|_{\hat{M}_{\epsilon }}
 =\left(\frac{\partial c_2}{\partial {\bf n}}-\alpha_2c_2\frac{\partial \Phi}{\partial {\bf n}}\right)|_{\hat{M}_{\epsilon }} = 0,\]
 the above homogeneous Neumann conditions on $\hat{M}_{\epsilon}$ are reasonable (they  are the consequences of the non-flux and zero-outward electric field conditions).

   In this  paper,  we are interested in the  limiting behavior of the PNP system when the three-dimensional tubular-like domain $\Omega_\epsilon$ collapses to a   one-dimensional interval as $\epsilon\to 0$. Naturally we expect a one-dimensional limiting system whose global dynamics is comparable with those of PNP systems for $\epsilon>0$ small. This important idea was applied by many researchers in studying the dynamics of equations defined on
 thin domains (see, e.g.,  \cite{HR1, HR2, RS1, TZ1}). We follow the procedure in \cite{HR2} to derive a one-dimensional limiting system but avoid expressing differential operators and transformations in local coordinates. As a result, the expected one-dimensional limiting system is  more transparent.

  To derive the limiting PNP system, we   transfer
 the $\epsilon$-dependent domain $\Omega_{\epsilon}$
 into a fixed domain  $\Omega=[0,1]\times \bbD$,  where    $\bbD$
 is   the unit disk, by applying  the following change of  coordinates:
\begin{equation}\label{chvar}
 x=X,\;y = {\frac Y{g(X,\epsilon)}},\; z={\frac Z{g(X,\epsilon)}}.
\end{equation}

In the sequel,  we denote by $L$, $R$ and $M$, respectively, the  boundaries of $\Omega$  corresponding to $\hat{L}_{\epsilon}$, $\hat{R}_{\epsilon}$ and $\hat{M}_{\epsilon}$ under the transformation. Let $J$ denote the Jacobian matrix of the change of coordinates. Then,
\[J=\frac{\partial(x,y,z)}{\partial(X,Y,Z)}=
\frac{1}{g^2}\left(\begin{array}{ccc}
g^2 & 0 & 0\\
-gg_xy & g  &  0\\
-gg_xz &  0 & g
\end{array}\right),\quad
J^{-1}=\frac{\partial(X,Y,Z)}{\partial(x,y,z)}=
\left(\begin{array}{ccc}
1 & 0 & 0\\
g_xy & g  &  0\\
g_xz &  0& g
\end{array}\right)\]
with  $\det(J^{-1})=g^2(x,\epsilon)$, and
\[JJ^{\tau}=
\frac{1}{g^4}\left(\begin{array}{ccc}
g^4 & -g^3g_xy & -g^3g_xz\\
-g^3g_xy   & g^2+g^2g_x^2y^2  &  g^2g_x^2yz\\
 -g^3g_xz & g^2g_x^2yz  & g^2+g^2g_x^2z^2
\end{array}\right).\]
The following result, which can be verified by direct computations, is useful for a clean derivation of a limiting PNP system.
 \begin{lem}\label{coordinaterelation}
 Let $\psi:\bbR^n\to \bbR^n$, $\psi(p)=q$, be a diffeomorphism, and let
  $J(q)=\frac{\partial q}{\partial p}(\psi^{-1}(q))$ be the Jacobian matrix and $d(q)=(\det J(q))^{-1}$.
If $\alpha(p)=\beta(\psi(p)):\bbR^n\to \bbR$ is a smooth function,
  then the gradients   in the two coordinates are related as
 \[\nabla_p\alpha(p)=J^{\tau}(q)\nabla_q\beta(q).\]
 Further, if $\sum\limits_{j=1}^n {\frac {\partial}{\partial q_j}}
 \left (d(q) {\frac {\partial q_j}{\partial p_i}}\right ) =0$ for all $i=1, \cdots, n$, and
 $f: \bbR^n\to \bbR^n$ is a smooth vector field, then $F(p) = f(\psi(p))$ satisfies
 \[  \nabla_p\cdot F(p)=\frac{1}{d(q)}\nabla_q\cdot\left(d(q)J(q)f(q)\right),\]
  and  hence, the Laplace operators are related as
\[\Lap_p\alpha(p)=\frac{1}{d(q)}\nabla_q\cdot
 \left(d(q)J(q)J^{\tau}(q)\nabla_q\beta(q)\right).\]
   \end{lem}

It can be checked that the change of variables in \eqref{chvar} with $p=(X,Y,Z)$ and $q=(x,y,z)$ satisfies
 \[\sum\limits_{j=1}^n {\frac {\partial}{\partial q_j}}
 \left (d(q) {\frac {\partial q_j}{\partial p_i}}\right ) =0.\]
  Therefore,
 applying Lemma \ref{coordinaterelation},  system~\eqref{pnp} can be rewritten,  in terms of $(x,y,z)$, as follows.
\begin{align}\label{newPNP}\begin{split}
\frac{1}{g^2}&\nabla\cdot\left(g^2 JJ^{\tau}\nabla\Phi\right)
 =-\lambda (\alpha_1c_1-\alpha_2c_2), \\
  \frac{\partial c_1}{\partial t}&= \frac{D_1}{g^2}\nabla \cdot
   (g^2JJ^{\tau}\nabla c_1 +\alpha_1c_1g^2 JJ^{\tau}\nabla \Phi),\\
    \frac{\partial c_2}{\partial t}&= \frac{D_2}{g^2} \nabla \cdot
 (g^2JJ^{\tau} \nabla c_2 - \alpha_2 c_2 g^2 JJ^{\tau}\nabla \Phi),
     \end{split}
 \end{align}
 with  the boundary conditions
\begin{align}\label{newBV}\begin{split}
\Phi|_{L } =\phi_0,   \quad  \Phi|_{R}=0, & \quad
c_k|_{L} =l_k,  \quad c_k|_{R}=r_k,\\
 \langle \nabla \Phi, JJ^{\tau}{\bf \nu}\rangle |_{M }
 &= \langle \nabla c_k, JJ^{\tau}{\bf \nu}\rangle |_{M }= 0,
 \end{split}
\end{align}
where $k=1,2$ and  ${\bf \nu}$ is the outward unit normal vector  to $M$.

By inspecting the structural dependence of $JJ^{\tau}$ on
 $\epsilon$, we expect the one-dimensional limiting  PNP system to be
 \begin{align}\label{genlimPNP}
\begin{split}
 & {\frac 1{g^2_0}} {\frac {\partial}{\partial x}} \left (
g_0^2 {\frac {\partial}{\partial x}} \Phi \right )
 =-\lambda(\alpha_1 c_1-\alpha_2 c_2),\\
& \frac{\partial c_1}{\partial t}
   = \frac{D_1}{g_0^2} {\frac {\partial}{\partial x}}
     \left( g_0^2  {\frac {\partial}{\partial x}}c_1+\alpha_1c_1g_0^2
  {\frac {\partial}{\partial x}} \Phi\right),\\
   &   \frac{\partial c_2}{\partial t}= \frac{D_2}{g_0^2}
  {\frac {\partial}{\partial x}}
    \left(g_0^2 {\frac {\partial}{\partial x}} c_2
  -\alpha_2 c_2 g_0^2{\frac {\partial}{\partial x}}\Phi \right),
 \end{split}\end{align}
on $x\in (0,1)$ with the   boundary conditions
\begin{equation}\label{genlimBV}
 \Phi(t,0) =\phi_0,   \quad  \Phi(t,1)=0,   \quad
c_k(t,0) =l_k,  \quad c_k(t,1)=r_k,
\end{equation}
 where $g_0(x)$ is defined in \eqref{gC}.

It was shown in \cite{GG2} that, for any $\epsilon>0$,   the  three-dimensional PNP system has a global attractor  ${\mathcal{A}}_\epsilon$ which is a compact subset and attracts all solutions with respect to the norm topology of $H^1 \times H^1$. This result is based on an invariant principle discovered in   \cite{Gro1, Gro2, GG2, Seid1} for the van Roosbroeck models of semi-conductor.    The PNP systems are basically the same as the van Roosbroeck models and  we   recall  the  invariant principle  using the above setting.

\begin{prop}
\label{invariantregion} Let $M$ be a positive constant with
 \[ M \ge \max \{ \alpha_1 l_1,
\alpha_1 r_1, \alpha_2 l_2, \alpha_2 r_2 \},
\]
and let ${\tilde{\Sigma}}$ be the subset of $H^1(\Omega_{\epsilon})
\times H^1(\Omega_{\epsilon})$ given by
\[\tilde{\Sigma} = \{ (c_1, c_2) \in H^1(\Omega_{\epsilon}) \times
H^1(\Omega_{\epsilon}): \ 0 \le \alpha_1 c_1 \le M, \ 0\le
\alpha_2 c_2 \le M \}.\]
Then   ${\tilde{\Sigma}}$ is positively invariant for the
PNP system. More precisely, if the initial datum
$(c_1(0), c_2(0))\in {\tilde{\Sigma}}$ and $(c_1, c_2)$ is the solution of the PNP system,
 then $(c_1(t), c_2(t)) \in {\tilde{\Sigma}}$ for all $t \ge 0$.
\end{prop}
We remark that,  for   PNP systems with three or more types of ions, the above invariant principle is not available.
It    is not clear to us whether or not a similar principle  still holds in this case.
  PNP systems with more than  two types of ions are worth further studying.

The results in \cite{Gro1, Gro2, GG2, Seid1} show also  that the one-dimensional problem \eqref{genlimPNP}-\eqref{genlimBV}   has a positively invariant region
\begin{equation*}
  {\tilde{\Sigma}}_0 = \{ (c_1, c_2) \in H^1(0,1) \times H^1(0,1): \
0 \le \alpha_1 c_1 \le M, \
  0  \le \alpha_2 c_2 \le M \},
  \end{equation*}
  where $M$ is the constant   in Proposition~\ref{invariantregion}, and
  problem \eqref{genlimPNP}-\eqref{genlimBV} is  globally well-posed in ${\tilde{\Sigma}}_0$
 and has a global attractor ${\mathcal{A}}_0$
in ${\tilde{\Sigma}}_0 \cap H^1(0,1) \times H^1(0,1)$.

Our first result claims that the global attractors
 ${\mathcal{A}}_\epsilon$ of  the three-dimensional
 PNP systems are upper semi-continuous to the
 global attractor ${\mathcal{A}}_0$ of the
 one-dimensional limiting system as $\epsilon \to 0$, which partially
 justify the limiting process.

 \begin{thm}
 \label{uppersc}
  The global attractors ${\mathcal A}_{\epsilon}$ of the three-dimensional PNP systems are
upper semi-continuous at $\epsilon=0$,  that is, for any $\eta>0$, there  exists  a
positive number $\epsilon_1=\epsilon_1(\eta)$ such that
 for all
$0<\epsilon\le\epsilon_1$ and all $w\in {\mathcal A}_{\epsilon}$,
\[{\rm dist}_{X_{\epsilon}} \left(w,  {\mathcal  A}_0 \right )\le \eta,\]
 where  $ {X_{\epsilon}} = \left\{   w: \   \| w \|^2_{X_\epsilon} =  \|w\|_{L^2}^2 + \| \nabla w \|_{L^2}^2 +\frac{1}{\epsilon^2}\|w_y\|^2_{L^2}
  +\frac{1}{\epsilon^2}\|w_z\|^2_{L^2} < \infty  \right\}$.
   \end{thm}

\subsection{Steady-state problem of the one-dimensional  limiting PNP system}
 The steady-state of problem~\eqref{genlimPNP} and~\eqref{genlimBV} can be rewritten as
\begin{align}\label{SSPNP}\begin{split}
\mu^2   \frac{d}{dx}\left(h(x)\frac{d\phi}{dx}\right)
  = -h(x) &(\alpha c_1-\beta c_2), \quad
    \frac{d J_1}{dx} =0,\quad \frac{d J_2}{dx} =0,\\
  h(x) \frac{d c_1}{dx}+\alpha c_1 h(x) \frac{d \phi}{dx}  = &
- J_1 ,\quad
  h(x)  \frac{d c_2}{dx} -\beta c_2 h(x) \frac{d \phi}{dx} = -  J_2
\end{split}
\end{align}
where $\mu^2=1/{\lambda}$,  $J_1=\bar{J}_1/{D_1}$,  $J_2=\bar{J}_2/{D_2}$ and $h(x)=g_0^2(x)$,
and the boundary conditions are
\begin{align}\label{SSBV}\begin{split}
\phi(0)=\phi_0,  & \quad c_1(0)=l_1,\quad c_2(0)=l_2,\\
\phi(1)=0,  & \quad c_1(1)=r_1, \quad   c_2(1)=r_2.
\end{split}
\end{align}
Since $\lambda$ is large, we can treat the problem~\eqref{SSPNP} and~\eqref{SSBV} as a singularly perturbed problem with $\mu$ as the singular parameter. We will recast the singularly perturbed PNP system into a system of first
order equations.

Denote derivatives with respect to $x$ by overdot   and
introduce
\[  \tau=x,\quad u =\mu h(\tau) \dot \phi ,\quad v=-h(\tau) (\alpha_1 c_1-\alpha_2 c_2),\;
\mbox{ and }\; w=\alpha_1^2 c_1+\alpha_2^2 c_2.\]
  System~\eqref{SSPNP} becomes
 \begin{align}\label{slow}\begin{split}
  \mu\dot \phi=& \frac{1}{h(\tau)}u,\quad
  \mu \dot u=  v, \quad
  \mu \dot v=uw +\mu\frac{h_{\tau}(\tau)}{h(\tau)}v+\mu(\alpha_1 J_1-\alpha_2 J_2),\\
  \mu \dot w=&\frac{\alpha_1\alpha_2}{h^2(\tau)} uv+\frac{\alpha_2-\alpha_1}{h(\tau)}uw
     -\frac{\mu}{h(\tau)}(\alpha_1^2 J_1+\alpha_2^2 J_2),\\
\dot J_1=& 0, \quad  \dot J_2=0,\quad \dot \tau=1.
 \end{split}
 \end{align}
 System~\eqref{slow} -- {\em the slow system} -- will be treated as a dynamical system   with the
 phase space
 \[\bbR^7=\{(\phi,u,v,w,J_1,J_2,\tau)\}\]
  and the independent variable $x$ will be viewed  as time. The boundary condition~\eqref{SSBV}  becomes
\begin{align}\label{scaledbv}\begin{split}
\phi(0)=&\phi_0,   \;  v(0)=-h(0)(\alpha_1 l_1-\alpha_2 l_2),
     \; w(0)=\alpha^2 L_1+\beta^2 L_2, \; \tau(0)=0,\\
\phi(1)=&0,  \;  v(1)=-h(1)(\alpha_1 r_1-\alpha_2 r_2),
    \;  w(1)=\alpha_1^2 r_1+\alpha_2^2 r_2, \; \tau(1)=1.
\end{split}
\end{align}

Setting $\mu=0$ in system~\eqref{slow}, we get the limiting slow system
\begin{align}\label{limslow}\begin{split}
  0=& \frac{1}{h(\tau)}u,\quad
  0=  v, \quad
   0=uw,\\
  0=&\frac{\alpha_1\alpha_2}{h^2(\tau)} uv+\frac{\alpha_2-\alpha_1}{h(\tau)}uw,\\
\dot J_1=& 0, \quad  \dot J_2=0,\quad \dot \tau=1.
 \end{split}
 \end{align}
The set ${\mathcal Z}_0=\{u=v=0\}$ is called {\em the slow manifold} which
supports the regular layer of the boundary value problem.  The regular layer
will not satisfy all conditions in~\eqref{scaledbv} if
 $\alpha_2 l_2-\alpha_1 l_1\neq 0$ or $\alpha_2 r_2-\alpha_1 r_1\neq 0$,
and this defect has to be remedied by boundary layers.
The boundary layer behavior will be determined by the
  fast system resulting from the slow system~\eqref{slow} by  the rescaling of time $x=\mu \xi$.
Thus, in terms of $\xi$,   {\em the fast system} of~\eqref{slow} is
 \begin{align}\label{fast}\begin{split}
  \phi'=& \frac{1}{h(\tau)} u,\quad u'=  v,\quad
   v'=uw+\mu\frac{h_{\tau}(\tau)}{h(\tau)}v+\mu(\alpha_1 J_1-\alpha_2 J_2),\\
    w'=&\frac{\alpha_1\alpha_2}{h^2(\tau)} uv+\frac{\alpha_2-\alpha_1}{h(\tau)}uw
     -\frac{\mu}{h(\tau)}(\alpha_1^2 J_1+\alpha_2^2 J_2),\\
 J_1'=& 0, \quad   J_2'=0,\quad  \tau'=\mu.
 \end{split}
 \end{align}
where   prime denotes the derivative with respect to the variable
$\xi$. The limiting fast system at  $\mu= 0$ is
\begin{align}\label{limfast}\begin{split}
  \phi'=& \frac{1}{h(\tau)} u,\quad
   u'=  v,\quad
   v'=uw, \quad
   w'=\frac{\alpha_1\alpha_2}{h^2(\tau)} uv+\frac{\alpha_2-\alpha_1}{h(\tau)}uw, \\
   J_1'=&0, \quad J_2'=0,\quad \tau'= 0.
 \end{split}
 \end{align}
The slow manifold ${\mathcal Z}_0$ is precisely the
set of equilibria of~\eqref{limfast}.

Concerning the steady-state problem of the one-dimensional limiting PNP system, we have
\begin{thm}\label{main} Assume that $\alpha_1 l_1\neq   \alpha_2 l_2$
and $\alpha_1 r_1\neq   \alpha_2 r_2$ (otherwise, see Remark \ref{nolayer}).
For $\mu>0$ small, the boundary value problem \eqref{slow} and \eqref{scaledbv} has a unique solution near a singular orbit. The
singular orbit is the union of  two  fast orbits of system~\eqref{limfast} representing the boundary layers and one slow orbit of a blow-up system of~\eqref{limslow} (see Section 4 for details) for the regular layer. The limiting   flux densities are explicitly given by
\begin{align*}
\bar{J}_1=&D_1J_1=\frac{D_1\left(\ln\frac{r_1}{l_1}-\alpha_1 \phi_0\right)
\left((\alpha_1 l_1)^{\frac{\alpha_2}{\alpha_1+\alpha_2}}
 (\alpha_2l_2)^{\frac{\alpha_1}{\alpha_1+  \alpha_2}}
 - (\alpha_1 r_1)^{\frac{  \alpha_2}{\alpha_1+  \alpha_2}}
 (  \alpha_2 r_2)^{\frac{\alpha_1}{\alpha_1+  \alpha_2}}\right)}
 {\left(\frac{\alpha_1 \alpha_2 }{\alpha_1+  \alpha_2}\ln\frac{r_1}{l_1}+
\frac{\alpha_1^2}{\alpha_1+  \alpha_2}\ln\frac{r_2}{l_2}\right)\int_0^1h^{-1}(x)\,dx},\\
\bar{J}_2=&D_2J_2=\frac{D_2\left(\ln\frac{r_2}{l_2}+  \alpha_2\phi_0\right)
\left((\alpha_1 l_1)^{\frac{  \alpha_2}{\alpha_1+  \alpha_2}}
 (  \alpha_2 l_2)^{\frac{\alpha_1}{\alpha_1+  \alpha_2}}
 - (\alpha_1 r_1)^{\frac{ \alpha_2}{\alpha_1+  \alpha_2}}
 (  \alpha_2 r_2)^{\frac{\alpha_1}{\alpha_1+  \alpha_2}}\right)}
{\left(\frac{  \alpha_2^2}{\alpha_1+  \alpha_2}\ln\frac{r_1}{l_1}+
\frac{\alpha_1 \alpha_2 }{\alpha_1+  \alpha_2}\ln\frac{r_2}{l_2}\right)\int_0^1h^{-1}(x)\,dx}.
\end{align*}
 \end{thm}

\begin{rem}
Note that the factor $\int_0^1h^{-1}(x)\,dx$ on the denominators in the expressions for the flux densities $J_1$ and $J_2$ reflects the effect of the geometry of the three-dimensional channel.   Let's compare this effect with that of a cylindrical channel where the wall is defined by  $\{Y^2+Z^2= \epsilon\}$.  In this case, the corresponding integral factor on  the denominators   for the flux densities $J_1$ and $J_2$ is $1$.  The volume of the channel is $\pi \epsilon^2$. For general channels that we considered here, we thus assume
\[\mbox{Vol}=\int_0^1\pi g^2(x;\epsilon)\,dx=\pi \epsilon^2.\]
From which we have $\int_0^1 h(x)\,dx=1$. Therefore,
\[1=\left(\int_0^1h^{-1/2}(x)h^{1/2}(x)\,dx\right)^2\le \int_0^1h^{-1}(x)\,dx\int_0^1h(x)\,dx=
 \int_0^1h^{-1}(x)\,dx.\]
The inequality indicates  that {\em the more complicated the geometry of the channel, the smaller the flux for the ion flow}, which    agrees with  our  common  sense.
\end{rem}

\section{Upper semi-continuity of attractors}
\setcounter{equation}{0}
\subsection{Homogenization of  boundary conditions}
 In this section,  we  convert the non-homogeneous   Dirichlet boundary conditions on $L\cup R$ in \eqref{newBV} to   homogeneous ones, while  keeping  the homogeneous Neumann boundary conditions on $M$. For this purpose, the following technical result is needed.

 \begin{lem}\label{foliation} Let $h: [0,1]\to \bbR$ be a smooth function. Then, for any $\epsilon>0$,
there is a function $H^{\epsilon}:\Omega_{\epsilon}\to \bbR$ such that $H^{\epsilon}(X,0,0)=h(X)$,
$H^{\epsilon}(0,Y,Z)=h(0)$,   $H^{\epsilon}(1,Y,Z)=h(1)$, and $\langle \nabla H^{\epsilon}(X,Y,Z), {\bf n}\rangle=0$ for $(X,Y,Z)\in \hat{M}_{\epsilon}$.
\end{lem}
\begin{proof} We provide a specific construction of a function $H^{\epsilon}$.
For convenience, hereafter, we denote by $g'(X, \epsilon) = {\frac {\partial g}{\partial X}}(X, \epsilon)$.
For any $\epsilon>0$ and $X_0\in [0,1]$, let $X=\psi^{\epsilon}(t,X_0)$ be the solution of
\begin{equation}\label{ode}
\frac{dX}{dt}=-t \frac{g'(X,\epsilon)}{g(X,\epsilon)}
\end{equation}
 with $\psi^{\epsilon}(0,X_0)=X_0$. It is easy to see that $\psi^{\epsilon}(t,X_0)$ is even in $t$ from the equation. Since $g'(0,\epsilon)=g'(1,\epsilon)=0$, $\psi^{\epsilon}(t,0)=0$ and $\psi^{\epsilon}(t,1)=1$ for all $t$. Therefore,
 for any $(X,t)\in [0,1]\times [0,g(X,\epsilon)]$, there is a unique $X_0\in [0,1]$ such that
 $X=\psi^{\epsilon}(t,X_0)$, and hence, for any $(X,Y,Z)\in \Omega_{\epsilon}$, there is a unique $X_0\in [0,1]$ such that $X=\psi^{\epsilon}(\sqrt{Y^2+Z^2},X_0)$.  Set
  $H^{\epsilon}(X,Y,Z)=h(X_0)$ if $X=\psi^{\epsilon}(\sqrt{Y^2+Z^2},X_0)$. Then,
  $H^{\epsilon}(X,0,0)=h(X)$, $H^{\epsilon}(0,Y,Z)=h(0)$ and  $H^{\epsilon}(1,Y,Z)=h(1)$. It remains to show that, for $(X,Y,Z)\in \hat{M}_{\epsilon}$, $\langle \nabla H^{\epsilon}(X,Y,Z), {\bf n}\rangle=0$.
 For any $X_0\in [0,1]$, the set
 \[D(X_0)=\{(X,Y,Z): X=\psi^{\epsilon}(\sqrt{Y^2+Z^2},X_0) \}=\{(X,Y,Z): H(X,Y,Z)=h(X_0)\},\]
 is a level set of $H^{\epsilon}$.  Note also that the curve $\{(X,Y,0): X=\psi^{\epsilon}(Y,X_0)\}$ lies on $D(X_0)$ and it is a solution curve to~\eqref{ode} if $Y$ is viewed as the $t$-variable.   Therefore, at $(X,Y,0)=(X,g(X,\epsilon),0)\in  D(X_0)\cap \hat{M}_{\epsilon}$, the vector
 \[\left(-Y\frac{g'(X,\epsilon)}{g(X,\epsilon)},1,0\right)=(-g'(X,\epsilon),1,0)\]
 is tangent to $D(X_0)$, and hence,  $\langle \nabla H^{\epsilon}(X,g(X,\epsilon),0), (-g'(X,\epsilon),1,0)\rangle=0$. Since ${\bf n}$ is parallel to $(-g'(X,\epsilon),1,0)$, $\langle \nabla H^{\epsilon}(X,g(X,\epsilon),0), {\bf n}\rangle=0$. Due to the rotation symmetry of $\hat{M}_{\epsilon}$ and $H^{\epsilon}$ about the $X$-axis, we conclude that,
 for $(X,Y,Z)\in \hat{M}_{\epsilon}$, $\langle \nabla H^{\epsilon}(X,Y,Z), {\bf n}\rangle=0$.
 \end{proof}

Let $L_k^0(X)$, for $k=1,2,3$, be the linear functions satisfying
$L_k^0(0)=l_k$, $L_k^0(1)=r_k$ for $k=1,2$, $L_3^0(0)=\phi_0$ and $L_3^0(1)=0$.
 Lemma \ref{foliation}  guarantees the existence of functions $L_k(X,Y,Z,\epsilon)$ for $k=1,2,3$   such that for each $\epsilon>0$ and  $Y^2 + Z^2 < g^2(X,\epsilon)$,
$L_k (X,0,0,\epsilon)=L^0_k$,  $L_k(0,Y,Z,\epsilon)=L_k^0(0)$, $L_k(1,Y,Z,\epsilon)=L_k^0(1)$,   and
${\frac {\partial  }{\partial {\bf n}}}L_k(X,Y,Z,\epsilon) =0$
 when $(X,Y,Z)\in \hat{M}_{\epsilon}$. For each $\epsilon>0$ and $k=1,2,3$,  introduce the functions $L^\epsilon_k$ in terms of variables $x, y$ and $z$:
 \begin{equation}\label{AF} L^\epsilon_k (x,y,z) = L_k(X,Y,Z,\epsilon)= L_k(x, g(x,\epsilon)y, g(x,\epsilon)z,\epsilon).
 \end{equation}

 Set
 \begin{align*}
 u(x,y,z)= & L^\epsilon_1(x,y,z)-c_1(x,y,z), \\
  v(x,y,z)= & L^\epsilon_2(x,y,z)-c_2(x,y,z),\\
 \phi(x,y,z)=& L^\epsilon_3(x,y,z)-\Phi(x,y,z).
 \end{align*}
 Then, problem \eqref{newPNP}-\eqref{newBV} is transformed into
\begin{align}\label{PNP}\begin{split}
\frac{1}{g^2}&\nabla\cdot\left(g^2JJ^{\tau}\nabla(\phi-L^\epsilon_3)\right)
   =-\lambda (\alpha_1(u-L^\epsilon_1)-\alpha_2(v-L^\epsilon_2)), \\
  \frac{\partial u}{\partial t}&= \frac{D_1}{g^2}\nabla \cdot
   \left(g^2JJ^{\tau}\nabla  (u-L^\epsilon_1)
     - \alpha_1(u-L^\epsilon_1(x))g^2JJ^{\tau}\nabla (\phi-L^\epsilon_3)\right),\\
    \frac{\partial v}{\partial t}&= \frac{D_2}{g^2} \nabla \cdot
 \left(g^2JJ^{\tau}\nabla (v-L^\epsilon_2)
 +\alpha_2(v-L^\epsilon_2(x))g^2JJ^{\tau}\nabla (\phi-L^\epsilon_3)\right),
 \end{split}
 \end{align}
 with the homogeneous   boundary conditions:
\begin{align}\label{BV}\begin{split}
\phi|_{L\cup R} = &u|_{L \cup R }=v|_{L \cup R }=0, \\
 \langle \nabla \phi, JJ^{\tau}\nu\rangle |_{M }
 = & \langle \nabla u, JJ^{\tau}\nu\rangle |_{M }=
 \langle \nabla v, JJ^{\tau}\nu\rangle |_{M }= 0.
\end{split}
\end{align}
  System \eqref{PNP} is supplemented with the initial conditions:
  \begin{equation}
  \label{IC}
  u(0) =u_0, \quad v(0)=v_0.
  \end{equation}
Introduce the subspace $H^1_D(\Omega)$ of $H^1(\Omega)$:
\[H^1_D(\Omega) =  \{ u \in H^1(\Omega): \ u |_{L \cup R} =0 \}.\]
 Let $M$ be the constant in Proposition \ref{invariantregion} and let
$\Sigma_\epsilon$ be the subset of $H^1_D(\Omega) \times
H^1_D(\Omega)$ given by
\begin{equation}
\label{IR}
  \Sigma_\epsilon = \{ (u,v) \in H^1_D(\Omega) \times H^1_D(\Omega): \
  \alpha_1 L^\epsilon_1 -M \le \alpha_1 u \le \alpha_1 L^\epsilon_1, \
  \alpha_2 L^\epsilon_2 -M \le \alpha_2 v \le \alpha_2 L^\epsilon_2 \}.
  \end{equation}
It follows from Proposition \ref{invariantregion} that, if
$(u_0, v_0)\in \Sigma_{\epsilon}$, then $(u(t), v(t)) \in { {\Sigma}}_\epsilon$
for every $t \ge 0$.  Throughout this paper,  for every  $\epsilon>0$,  we denote by
$S^{\epsilon}(t)_{t \ge 0}$ the solution operator associated with problem
\eqref{PNP}-\eqref{IC}.
We will use the same symbol ${\mathcal{A}}_\epsilon$ to denote the global attractors
 of $S^{\epsilon}(t)_{t \ge 0}$ and that of problem \eqref{newPNP}-\eqref{newBV} when no confusion arises.

 The corresponding one-dimensional limiting system~\eqref{genlimPNP} is transformed into \begin{align}\label{limPNP}\begin{split}
 & {\frac 1{g^2_0}} {\frac {\partial}{\partial x}}
 \left (g_0^2 {\frac {\partial}{\partial x}} \left ( \phi -L_3^0 \right )\right )
 =-\lambda \left(\alpha_1(u-L^0_1)-\alpha_2(v-L^0_2)\right),\\
& \frac{\partial u}{\partial t}= \frac{D_1}{g_0^2}
  {\frac {\partial}{\partial x}}\left( g_0^2  {\frac {\partial}{\partial x}}   (u-L^0_1)
  -\alpha_1(u-L^0_1)g_0^2{\frac {\partial}{\partial x}} (\phi-L^0_3)\right),\\
   &   \frac{\partial v}{\partial t}= \frac{D_2}{g_0^2}
  {\frac {\partial}{\partial x}}\left(g_0^2 {\frac {\partial}{\partial x}}  (v-L^0_2)
  +\alpha_2(v-L^0_2)g_0^2  {\frac {\partial}{\partial x}} (\phi-L^0_3) \right),
 \end{split}\end{align}
with the homogeneous Dirichlet boundary conditions
\begin{equation}\label{limBV}
\phi=  u= v=0, \quad x =0,1,
\end{equation}
and the initial conditions
 \begin{equation}
 \label{limIC}
u(0) = u_0, \quad {\rm and} \ v(0) =v_0.
\end{equation}

Since ${\tilde{\Sigma}}_0$  is an invariant region
for problem \eqref{genlimPNP}-\eqref{genlimBV}, we find that
  the one-dimensional problem \eqref{limPNP}-\eqref{limIC} also has a positively invariant region which is given by
\begin{equation}
\label{limIR}
  \Sigma_0 = \{ (u,v) \in H^1_0(0,1) \times H^1_0(0,1): \
  \alpha_1 L^0_1 -M \le \alpha_1 u \le \alpha_1 L^0_1, \
  \alpha_2 L^0_2 -M \le \alpha_2 v \le \alpha_2 L^0_2 \}.
  \end{equation}
  Similar to  system \eqref{genlimPNP}-\eqref{genlimBV},
    problem \eqref{limPNP}-\eqref{limIC} is      well-posed in $\Sigma_0$, that is, for each $(u_0, v_0) \in \Sigma_0$, there exists a unique solution $(u, v)$ for
problem \eqref{limPNP}-\eqref{limIC}   which is defined for all $t \ge 0$ and
$(u, v) \in {\mathcal{C}}([0, \infty), \Sigma_0)$. Further, the solutions are continuous in initial data with respect to  the topology of $H^1_0(0,1) \times H^1_0(0,1)$.
Therefore,  there  is a  continuous dynamical system $S^0(t)_{t\ge 0}$ associated with
problem \eqref{limPNP}-\eqref{limIC} such that for each $t \ge 0$
and $(u_0, v_0) \in \Sigma_0$, $S^0(t) (u_0, v_0) = (u(t), v(t))$,
the solution of problem \eqref{limPNP}-\eqref{limIC}.
When no  confusion arises, we use the same symbol
${\mathcal{A}}_0$ to denote the global attractors of
$S^0(t)_{t\ge 0}$ and problem  \eqref{genlimPNP}-\eqref{genlimBV}.

\subsection{Uniform estimates of global attractors}
   In this section, we derive  uniform
estimates of the global attractors ${\mathcal{A}}_\epsilon$ in
$\epsilon$ which are necessary for establishing the upper semi-continuity of ${\mathcal{A}}_\epsilon$ at $\epsilon =0$. In
what follows, we reformulate problem \eqref{PNP}-\eqref{IC} as an
abstract differential equation in $H^1_D(\Omega) \times
H^1_D(\Omega)$.

Given $\epsilon >0$, define an inner product $(\cdot,\cdot)_{H_{\epsilon}}$
 on $L^2(\Omega)$   by
\[(v,w)_{H_{\epsilon}}=\int_{\Omega}\frac{g^2}{\epsilon^2}vw\,dx\,dy\,dz,\]
and a bilinear form $a_{\epsilon}(\cdot,\cdot)$ on $\left(H^1_{D} (\Omega)\right)^2$ by
 \[a_{\epsilon}(w_1,w_2)=(J^{\tau}\nabla w_1,
 J^{\tau}\nabla w_2)_{H_{\epsilon}}
 =\int_{\Omega}{\frac {g^2}{\epsilon^2}} J^{\tau}\nabla w_1 \cdot J^{\tau}\nabla
 w_2 \,dx\,dy\,dz.\]
  In the sequel, we    denote $\|w\|_p$ the standard norm of $w$ for $w\in L^p(\Omega)$ or $w\in L^p([0,1])$, $\|w\|_{H^s}$ the standard norm of $w$ for  $w\in H^s(\Omega)$ or $w\in H^s([0,1])$. Also,   denote $H_\epsilon$ the space
 $L^2(\Omega)$ with the inner product
   $(\cdot,\cdot)_{H_{\epsilon}}$, and $X_\epsilon$ the space
   $H^1_D(\Omega)$ with the norm
\[\|w\|_{X_{\epsilon}}=\left(\| \nabla w\|^2_2
  +\frac{1}{\epsilon^2}\|w_y\|^2_2
  +\frac{1}{\epsilon^2}\|w_z\|^2_2\right)^{1/2}.\]
  Since Poincare inequality holds in $H^1_{D} (\Omega)$,   the norm $\|w\|_{X_{\epsilon}}$ for  $w \in H^1_{D} (\Omega)$  is equivalent to the
  norm given by
  \[ \left(\| w\|^2_{H^1}
  +\frac{1}{\epsilon^2}\|w_y\|^2_2
  +\frac{1}{\epsilon^2}\|w_z\|^2_2\right)^{1/2}.\]
   Due to  assumption \eqref{gC}, there exist positive constants $C_1, C_2, C_3$ (independent of $\epsilon$)  and $\epsilon_1$   such that,   for all $0 < \epsilon \le \epsilon_1$ and $x \in (0,1)$,
  \begin{equation}
  \label{g_inequality}
  {\frac {|g_x|}{g}} \le C_1, \quad C_2 \le {\frac g\epsilon} \le C_3.
  \end{equation}
 Consequently, $\sqrt{a_{\epsilon}(w,w)}$ is equivalent to the norm $\| w \|_{X_\epsilon}$, that is,
 \begin{equation}
    \label{honebd}
  C_4\|w\|^2_{X_{\epsilon}}\le a_{\epsilon}(w,w)\le
  C_5 \|w\|^2_{X_{\epsilon}}
  \end{equation}
 for some constants $C_4$ and $C_5$ (independent of $\epsilon$).  It follows from \eqref{honebd} that for each $\epsilon>0$,  the
  triple $\{ H^1_D (\Omega), H_\epsilon, a_\epsilon(\cdot, \cdot)\}$ defines a unique unbounded operator ${\mathcal{L}}_\epsilon$ on $H^1_D(\Omega)$ with domain
  $D({\mathcal{L}}_\epsilon)$ in the following way: an element $u \in H^1_D(\Omega)$ belongs to $D({\mathcal{L}}_\epsilon)$  if   $a_\epsilon(u,v)$ is
continuous in $v \in H^1_D(\Omega)$ for the topology induced from
$H_\epsilon$ and $( {\mathcal{L}}_\epsilon u , v)_{H_\epsilon} =
a_\epsilon (u,v)$ for  $(u,v) \in D({\mathcal{L}}_\epsilon) \times H^1_D(\Omega)$.
In fact,
\[D({\mathcal{L}}_\epsilon) = \{u\in
H^1_D(\Omega): \ {\mathcal{L}}_\epsilon u \in H_\epsilon \},\]
and for every $u \in D({\mathcal{L}}_\epsilon)$,
\[{\mathcal{L}}_\epsilon u = - {\frac 1{g^2}} \nabla \cdot \left (
g^2 J J^\tau \nabla u\right ).\]
Since the operator ${\mathcal{L}}_\epsilon$ is self-adjoint on
$H_\epsilon$ and positive, the fractional power
${\mathcal{L}}_\epsilon^{  1/2}$ is well-defined with domain
$D({\mathcal{L}}_\epsilon^{1/2})= H^1_D(\Omega)$, and for
$u \in H^1_D(\Omega)$,
\[\|{\mathcal{L}}_\epsilon^{\frac 12} u\|^2_{H_\epsilon} = a_\epsilon (u,u).\]
 In view of  \eqref{honebd}   there exist $C_6$ and $C_7$ such that
 \begin{equation}\label{lhalfbd}
  C_6\|u \|_{X_{\epsilon}}\le \|{\mathcal{L}}_\epsilon^{\frac 12} u\|_{H_\epsilon} \le
  C_7 \|u\|_{X_{\epsilon}}.
  \end{equation}
  With the above notations, system \eqref{PNP} can be rewritten as
 \begin{align}\label{absPNP}\begin{split}
 {\mathcal{L}}_\epsilon \phi
 & = \lambda \alpha_1(u-L^\epsilon_1)- \lambda
 \alpha_2(v-L^\epsilon_2) -
\frac{1}{g^2}\nabla\cdot\left(g^2JJ^{\tau}\nabla L_3^\epsilon \right),\\
   \frac{\partial u}{\partial t} + D_1 {\mathcal{L}}_\epsilon u &= - \frac{D_1}{g^2}\nabla \cdot
   \left(g^2JJ^{\tau}\nabla  L^\epsilon_1
     + \alpha_1(u-L^\epsilon_1 )g^2JJ^{\tau}\nabla (\phi-L^\epsilon_3)\right),\\
    \frac{\partial v}{\partial t} +
 D_2 {\mathcal{L}}_\epsilon v
    &=  - \frac{D_2}{g^2} \nabla \cdot
 \left(g^2JJ^{\tau}\nabla   L^\epsilon_2
 -\alpha_2(v-L^\epsilon_2 )g^2JJ^{\tau}\nabla
 (\phi-L^\epsilon_3)\right).
 \end{split}
 \end{align}
By the construction of functions $L^{\epsilon}_k$ ($k=1,2,3$),
there exists $\epsilon_1 >0$ such that for any $0<\epsilon \le \epsilon_1$,
the following
uniform bounds in $\epsilon$ hold:
\begin{equation}
\label{unibdL}
 \|L_k^\epsilon \|_\infty +  \| L_k^\epsilon \|_{H_\epsilon}
 + \|J^\tau  \nabla L_k^\epsilon \|_{H_\epsilon}
 + \|JJ^\tau  \nabla L_k^\epsilon \|_{H_\epsilon}
 + \|\frac{1}{g^2}\nabla\cdot\left(g^2JJ^{\tau}\nabla L_k^\epsilon
 \right)\|_{H_\epsilon} \le C,
 \end{equation}
 where  $ C$ is independent of $\epsilon$.
 Then it follows  from
  the positive invariance of $\Sigma_\epsilon$ that
  there exists  a constant $C$ (independent of $\epsilon$) such that
 for any initial datum $(u_0, v_0) \in \Sigma_\epsilon$, the
 solution $(u, v)$ of problem \eqref{PNP}-\eqref{IC} satisfies,
 for all $t \ge 0$:
 \begin{equation}
 \label{Linfbd}
 \| u(t) \|_{\infty} + \| v(t) \|_{\infty} \le C
 \quad  {\rm and} \quad
 \| u(t) \|_{H_\epsilon} + \| v(t) \|_{H_\epsilon} \le C.
\end{equation}
Next, we start to derive uniform estimates of solutions in
$\epsilon$ in the space $H^1_D(\Omega) \times H^1_D(\Omega)$.

\begin{lem}\label{lemma31}  There exist  a
constant $C$ (independent of $\epsilon$)  and $\epsilon_1 >0$
such that for any  $0<\epsilon \le \epsilon_1$ and $(u_0,
v_0) \in \Sigma_\epsilon$, the solution $(u,v)$ of problem
\eqref{PNP}-\eqref{IC} satisfies, for all $t \ge 0$:
\[\int_t^{t+1} \left ( \| u(t) \|_{X_\epsilon} + \| v(t)
\|_{X_\epsilon} \right ) dt \le C.\]
\end{lem}

\begin{proof}
Taking  the inner product of the first equation in \eqref{absPNP}
with $\phi$ in $H_\epsilon$, we find that
\[ \| {\mathcal{L}}_\epsilon^{\frac 12} \phi \|^2_{H_\epsilon}
= \lambda \alpha_1(u-L^\epsilon_1, \phi)_{H_\epsilon}- \lambda
 \alpha_2(v-L^\epsilon_2, \phi)_{H_\epsilon}+
(J^{\tau}\nabla L_3^\epsilon, J^\tau \nabla \phi)_{H_\epsilon}.\]
By \eqref{unibdL} and \eqref{Linfbd} we have
\begin{align*}
 \| {\mathcal{L}}_\epsilon^{\frac 12} \phi \|^2_{H_\epsilon}
\le & \lambda \alpha_1 (\|u \|_{H_\epsilon} + \| L_1^\epsilon
\|_{H_\epsilon}) \| \phi \|_{H_\epsilon}
 + \lambda \alpha_2 (\|v
\|_{H_\epsilon} + \| L_2^\epsilon \|_{H_\epsilon}) \| \phi
\|_{H_\epsilon} + \| J^\tau \nabla L_3^\epsilon \|_{H_\epsilon} \|
{\mathcal{L}}_\epsilon^{\frac 12} \phi \|_{H_\epsilon}\\
\le & C\|{\mathcal{L}}_\epsilon^{\frac 12} \phi \|_{H_\epsilon} \le {\frac
12}\| {\mathcal{L}}_\epsilon^{\frac 12} \phi \|_{H_\epsilon}^2 +{\frac 12} C^2,
\end{align*}
which implies that
\begin{equation}
\label{pflemma311}
 \| {\mathcal{L}}_\epsilon^{\frac 12} \phi\|_{H_\epsilon} \le C.
\end{equation}
Now, taking the inner product of the second equation in
\eqref{absPNP} with $u$ in $H_\epsilon$, we get
\[{\frac 12} {\frac d{dt}} \|u \|^2_{H_\epsilon} + D_1 \|
{\mathcal{L}}_\epsilon^{\frac 12} u \|^2_{H_\epsilon} = D_1 \left
( J^\tau \nabla L_1^\epsilon, J^\tau \nabla u \right
)_{H_\epsilon} + D_1 \alpha_1 \left ((u-L_1^\epsilon) J^\tau
\nabla (\phi -L_3^\epsilon), J^\tau \nabla u \right)_{H_\epsilon}.\]
It follows from \eqref{unibdL}-\eqref{pflemma311} that the
right-hand side of the above is bounded by
\begin{align*}
C_1\| J^\tau \nabla L_1^\epsilon \|_{H_\epsilon} \|
{\mathcal{L}}_\epsilon^{\frac 12} u \|_{H_\epsilon} +&
D_1\alpha_1(\|u\|_\infty + \|L_1^\epsilon \|_\infty) (\|
{\mathcal{L}}_\epsilon^{\frac 12} \phi \|_{H_\epsilon} + \| J^\tau
\nabla L_3^\epsilon \|_{H_\epsilon}) \|
{\mathcal{L}}_\epsilon^{\frac 12} u \|_{H_\epsilon}\\
 \le & C \|
{\mathcal{L}}_\epsilon^{\frac 12} u \|_{H_\epsilon} \le {\frac
{1}2} D_1\| {\mathcal{L}}_\epsilon^{\frac 12} u \|^2_{H_\epsilon}
+C_1.
\end{align*}
Therefore,
\begin{equation}
\label{pflemma312}
  {\frac d{dt}} \|u \|^2_{H_\epsilon} + D_1 \|
{\mathcal{L}}_\epsilon^{\frac 12} u \|^2_{H_\epsilon} \le C_2.
\end{equation}
Similarly,
\begin{equation}
\label{pflemma313}
  {\frac d{dt}} \|v \|^2_{H_\epsilon} + D_2 \|
{\mathcal{L}}_\epsilon^{\frac 12} v \|^2_{H_\epsilon} \le C_3.
\end{equation}
 Hence, for all $t \ge 0$:
 \[{\frac d{dt}} \left (  \|u \|^2_{H_\epsilon} + \|v \|^2_{H_\epsilon}\right )
  + C_4 \left ( \|
{\mathcal{L}}_\epsilon^{\frac 12} u \|^2_{H_\epsilon} +\|
{\mathcal{L}}_\epsilon^{\frac 12} v \|^2_{H_\epsilon} \right ) \le
C_2 +C_3,\]
which,  along \eqref{lhalfbd} and \eqref{Linfbd}, implies Lemma
\ref{lemma31}.
\end{proof}

\begin{lem}
\label{lemma32}  There exist positive constants $\epsilon_1$ and
  $C$   such that for any $0< \epsilon \le \epsilon_1$ and  $(u_0,
v_0) \in \Sigma_\epsilon$, the solution $(u,v)$ of problem
\eqref{PNP}-\eqref{IC} satisfies, for all $t \ge 1$:
\[\|{\mathcal{L}}_\epsilon \phi (t) \|_{X_\epsilon} +  \| u(t) \|_{X_\epsilon}
+ \| v(t)\|_{X_\epsilon}    \le C.\]
\end{lem}

\begin{proof}
By \eqref{unibdL}, \eqref{Linfbd} and the first equation in
\eqref{absPNP} we get
\begin{equation}
 \label{pflemma321}
 \| {\mathcal{L}}_\epsilon \phi \|_{H_\epsilon} \le
 C \left (\| u\|_{H_\epsilon} + \| v\|_{H_\epsilon} +\|
 L_1^\epsilon\|_{H_\epsilon}+\| L_2^\epsilon\|_{H_\epsilon}+
 \|\frac{1}{g^2}\nabla\cdot\left(g^2JJ^{\tau}\nabla L_3^\epsilon \right)
  \|_{H_\epsilon} \right )
  \le C.
 \end{equation}
 Taking the inner product of the second equation in \eqref{absPNP}
 with ${\mathcal{L}}_\epsilon u$ in $H_\epsilon$,  we find
 \begin{align} \label{pflemma322}
 {\frac 12} {\frac d{dt}} \| {\mathcal{L}}_\epsilon^{\frac 12} u
 \|^2_{H_\epsilon} + D_1 \| {\mathcal{L}}_\epsilon  u
 \|^2_{H_\epsilon}
 = &-  \left ( {\frac {D_1}{g^2}} \nabla \cdot (g^2 JJ^\tau \nabla
 L_1^\epsilon),   {\mathcal{L}}_\epsilon  u
 \right )_{H_\epsilon} \nonumber\\
 &- \left ( \frac{D_1 \alpha_1}{g^2}\nabla \cdot
    (  (u-L^\epsilon_1 )g^2JJ^{\tau}\nabla (\phi-L^\epsilon_3) ),{\mathcal{L}}_\epsilon  u
 \right )_{H_\epsilon}.
 \end{align}
 By \eqref{unibdL}, the first term on the right-hand side of
 \eqref{pflemma322} is bounded by
\begin{equation}
\label{pflemma323} |\left ( {\frac {D_1}{g^2}} \nabla \cdot (g^2
JJ^\tau \nabla
 L_1^\epsilon),   {\mathcal{L}}_\epsilon  u
 \right )_{H_\epsilon}|
 \le D_1 \| {\frac { 1}{g^2}} \nabla \cdot (g^2
JJ^\tau \nabla
 L_1^\epsilon)\|_{H_\epsilon} \| {\mathcal{L}}_\epsilon  u
 \|_{H_\epsilon}
 \le {\frac 14}D_1\| {\mathcal{L}}_\epsilon  u
 \|_{H_\epsilon}^2 +C.
 \end{equation}
 For the second term on the right-hand side of \eqref{pflemma322},
 we have
 \begin{align}\label{pflemma324}
 - &\left ( \frac{D_1 \alpha_1}{g^2}\nabla \cdot
    (  (u-L^\epsilon_1 )g^2JJ^{\tau}\nabla (\phi-L^\epsilon_3) ),{\mathcal{L}}_\epsilon  u
 \right )_{H_\epsilon}\nonumber\\
 &= -D_1 \alpha_1 \left ( \nabla (u-L_1^\epsilon) \cdot JJ^\tau \nabla (\phi
 -L_3^\epsilon), {\mathcal{L}}_\epsilon  u
 \right )_{H_\epsilon}\nonumber\\
  & -D_1\alpha_1
 \left ((u-L_1^\epsilon) {\frac 1{g^2}} \nabla \cdot (g^2 JJ^\tau \nabla (\phi
 -L_3^\epsilon)),{\mathcal{L}}_\epsilon  u\right )_{H_\epsilon}.
 \end{align}
  Using  \eqref{unibdL} and \eqref{pflemma321}, the
  first term on the right-hand side of \eqref{pflemma324} is bounded by
 \begin{eqnarray}
 \label{pflemma325}
&&D_1 \alpha_1 |\left (
 \nabla (u-L_1^\epsilon) \cdot JJ^\tau \nabla (\phi
 -L_3^\epsilon), {\mathcal{L}}_\epsilon  u
 \right )_{H_\epsilon}| \nonumber\\
 &&\le
 D_1 \alpha_1 \| \nabla (u- L_1^\epsilon) \|_3
 \| {\frac {g^2}{\epsilon^2}} JJ^\tau \nabla (\phi -L_3^\epsilon)
 \|_6 \| {\mathcal{L}}_\epsilon u \|_2 \nonumber \\
 && \le C \| \nabla (u -L_1^\epsilon) \|_2^{\frac 12} \| \nabla
 (u-L_1^\epsilon)\|^{\frac 12}_{H^1}
 \| {\frac {g^2}{\epsilon^2}} JJ^\tau \nabla (\phi -L_3^\epsilon)
 \|_{H^1} \|{\mathcal{L}}_\epsilon u \|_2 \nonumber\\
 &&\le \left ( \|{\mathcal{L}}_\epsilon^{\frac 12} u \|_{H_\epsilon}
 + \| J^\tau \nabla L_1^\epsilon \|_{H_\epsilon}
  \right )^{\frac 12}
  \left (
  \|{\mathcal{L}}_\epsilon u \|_{H_\epsilon}
  + \| J^\tau \nabla L_1^\epsilon \|_{H_\epsilon}
  + \|{\frac {1}{g^2}} \nabla \cdot (g^2 JJ^\tau \nabla
  L_1^\epsilon)\|_{H_\epsilon}
  \right )^{\frac 12} \nonumber \\
  && \times
  \left (
   \|{\mathcal{L}}_\epsilon \phi \|_{H_\epsilon} + \|JJ^\tau
   \nabla L_3^\epsilon \|_{H_\epsilon}   + \|{\frac {1}{g^2}} \nabla \cdot (g^2 JJ^\tau \nabla
  L_3^\epsilon)\|_{H_\epsilon}
  \right )  \|{\mathcal{L}}_\epsilon u \|_{H_\epsilon}
  \nonumber\\
  && \le C \left (
 \|{\mathcal{L}}_\epsilon^{\frac 12} u \|_{H_\epsilon} +C
  \right )^{\frac 12}
  \left (
\|{\mathcal{L}}_\epsilon u \|_{H_\epsilon} +C
  \right )^{\frac 12}   \|{\mathcal{L}}_\epsilon u \|_{H_\epsilon}
  \nonumber \\
  && \le
  {\frac 18} D_1\|{\mathcal{L}}_\epsilon u \|_{H_\epsilon}^2
  + C \|{\mathcal{L}}_\epsilon^{\frac 12} u \|_{H_\epsilon}^2
  +C.
 \end{eqnarray}
  The second  term on the right-hand side of \eqref{pflemma324} can be estimated as \begin{eqnarray}
\label{pflemma326}
 &&D_1\alpha_1\left |\left (
 (u-L_1^\epsilon) {\frac 1{g^2}} \nabla \cdot (g^2 JJ^\tau \nabla (\phi
 -L_3^\epsilon)),
 {\mathcal{L}}_\epsilon  u
 \right )_{H_\epsilon}\right | \nonumber \\
&& \le D_1\alpha_1
  \left (\| u \|_{\infty}
 + \|L_1^\epsilon\|_{\infty} \right )
 \|  {\frac 1{g^2}} \nabla \cdot (g^2 JJ^\tau \nabla (\phi
 -L_3^\epsilon))\|_{H_\epsilon}
  \|{\mathcal{L}}_\epsilon  u \|
  _{H_\epsilon} \nonumber\\
  &&
  \le D_1\alpha_1
  \left (\| u \|_{\infty}
 + \|L_1^\epsilon\|_{\infty} \right )
  \left (
  \| {\mathcal{L}}_\epsilon \phi \|_{H_\epsilon}
  + \|  {\frac 1{g^2}} \nabla \cdot (g^2 JJ^\tau \nabla
  L_3^\epsilon)\|_{H_\epsilon} \right )
  \|{\mathcal{L}}_\epsilon  u \|
  _{H_\epsilon} \nonumber\\
  && \le
  C \|{\mathcal{L}}_\epsilon  u \|
  _{H_\epsilon}
  \le {\frac 18}D_1 \|{\mathcal{L}}_\epsilon  u \|^2
  _{H_\epsilon} +C.
\end{eqnarray}
Combining the estimates \eqref{pflemma324}-\eqref{pflemma326}, we obtain
\begin{equation}\label{pflemma327}
 | \left (\frac{D_1 \alpha_1}{g^2}\nabla \cdot
    (  (u-L^\epsilon_1 )g^2JJ^{\tau}\nabla (\phi-L^\epsilon_3) ),
    {\mathcal{L}}_\epsilon  u\right )_{H_\epsilon}|
 \le {\frac 14}D_1 \|{\mathcal{L}}_\epsilon  u \|^2_{H_\epsilon}
 + C\|{\mathcal{L}}_\epsilon^{\frac 12}  u \|^2_{H_\epsilon}+C.
  \end{equation}
   It follows from \eqref{pflemma322}, \eqref{pflemma323} and
  \eqref{pflemma327}  that, for all $t \ge 0$,
 \begin{equation}
\label{pflemma328}
  {\frac d{dt}} \| {\mathcal{L}}_\epsilon^{\frac 12} u
 \|^2_{H_\epsilon} + D_1 \| {\mathcal{L}}_\epsilon  u
 \|^2_{H_\epsilon}
 \le C_1 \| {\mathcal{L}}_\epsilon^{\frac 12} u
 \|^2_{H_\epsilon} +C_2.
 \end{equation}
 Similarly,   for all $t \ge 0$,
  \begin{equation}
\label{pflemma329}
  {\frac d{dt}} \| {\mathcal{L}}_\epsilon^{\frac 12} v
 \|^2_{H_\epsilon} + D_2 \| {\mathcal{L}}_\epsilon  v
 \|^2_{H_\epsilon}
 \le C_1 \| {\mathcal{L}}_\epsilon^{\frac 12} v
 \|^2_{H_\epsilon} +C_2.
 \end{equation}
 Hence,   we have, for all $t \ge 0$,
 \begin{equation}\label{pflemma3210}
  {\frac d{dt}}  \left ( \| {\mathcal{L}}_\epsilon^{\frac 12} u \|^2_{H_\epsilon}
 + \| {\mathcal{L}}_\epsilon^{\frac 12} v \|^2_{H_\epsilon}  \right )
 +C_3 \left (  \| {\mathcal{L}}_\epsilon  u \|^2_{H_\epsilon}
 +  \| {\mathcal{L}}_\epsilon  v\|^2_{H_\epsilon} \right )
 \le C_1 \left (  \| {\mathcal{L}}_\epsilon^{\frac 12} u\|^2_{H_\epsilon}
 + \| {\mathcal{L}}_\epsilon^{\frac 12} v\|^2_{H_\epsilon} \right ) +C_2,
 \end{equation}
 which, along with Lemma \ref{lemma31} and the uniform Gronwall's
 lemma, implies that, for all $t \ge 1$,
\[ \| {\mathcal{L}}_\epsilon^{\frac 12} u(t)
 \|^2_{H_\epsilon}
 + \| {\mathcal{L}}_\epsilon^{\frac 12} v(t)
 \|^2_{H_\epsilon}     \le C.\]
The above estimate and the first equation in \eqref{absPNP} conclude
the proof.
   \end{proof}

Applying Gronwall's lemma to  \eqref{pflemma3210} for $t \in (0,1)$, then   by Lemma
\ref{lemma32} and the first equation in \eqref{absPNP}
 we find that there exists $\epsilon_1>0$ such that,
for any $R>0$, there exists $K$ depending on $R$ such that for any
$0<\epsilon \le \epsilon_1$ and $(u_0, v_0) \in \Sigma_\epsilon$
with $\|(u_0, v_0) \|_{X_\epsilon \times X_\epsilon} \le R$,
the following holds:
\begin{equation}
\label{alltimeh1bd}
\| \mathcal{L}_\epsilon \phi (t)\|_{X_\epsilon} +
 \| u(t) \|_{X_\epsilon} + \| v(t) \|_{X_\epsilon} \le K, \quad
 {\rm for} \ t \ge 0.
 \end{equation}
An immediate consequence of Lemma \ref{lemma32}  also shows that
all the global attractors ${\mathcal{A}}_\epsilon$ are uniformly
bounded in $\epsilon$ in the space $H^1_D(\Omega) \times
H^1_D(\Omega)$, that is, the following statement is true.

\begin{prop}
\label{proposition33} There exist positive constants $\epsilon_1$
and $C$ such that for all $0<\epsilon \le \epsilon_1$ and $(u,v)
\in {\mathcal{A}}_\epsilon$, the following holds:
$$ \|(u,v) \|_{X_\epsilon \times X_\epsilon} \le C.$$
\end{prop}

The following is an analogue of Lemma \ref{lemma32} for the
limiting system \eqref{limPNP}-\eqref{limIC}.

\begin{lem}\label{lemma34}
 There exists  $C>0$   such that for any $(u_0, v_0)
\in \Sigma_0$, the solution $(u,v)$ of problem
\eqref{limPNP}-\eqref{limIC} satisfies, for all $t \ge 1$:
\[ \| u(t) \|_{H^1} + \| v(t)\|_{H^1}    \le C.\]
In addition, there exists $K$ depending on $R$ when $\|(u_0,
v_0)\|_{H^1 \times H^1} \le R$ such that for all $t \ge 0$:
\[ \| u(t) \|_{H^1} + \| v(t)\|_{H^1}    \le K.\]
\end{lem}

Next, we establish   estimates on time derivatives of  solutions
for both the three-dimensional system and the one-dimensional limiting system.
\begin{lem}\label{lemma35}
There exists $\epsilon_1>0$ such that for any $R>0$, there exists $K$ depending
only on $R$ such that for any $0<\epsilon \le \epsilon_1$ and $(u_0, v_0) \in \Sigma_\epsilon$
with $\|(u_0, v_0) \|_{X_\epsilon \times X_\epsilon} \le R$, the
solution $(u,v)$ of problem \eqref{PNP}-\eqref{IC} satisfies
\[t^2 \left ( \| {\mathcal{L}}_\epsilon  {\frac {\partial \phi}{\partial t}} \|^2_{H_\epsilon}
+ \|{\frac {\partial u}{\partial t}} \|^2_{H_\epsilon} +
\|{\frac {\partial v}{\partial t}} \|^2_{H_\epsilon} \right )
+ \int_0^t s^2 \left (\|{\frac {\partial u}{\partial s}} \|^2_{X_\epsilon} +
\|{\frac {\partial v}{\partial s}} \|^2_{X_\epsilon}\right ) ds
\le Ke^{Kt}, \quad t \ge 0.\]
\end{lem}

\begin{proof} Denote by
 \[{\tilde{\phi}} = {\frac {\partial \phi}{\partial t}}, \quad
 {\tilde{u}} = {\frac {\partial u}{\partial t}}, \quad
 {\tilde{v}} = {\frac {\partial v}{\partial t}}.\]
  Differentiating \eqref{absPNP} with respect to $t$, we get
  \begin{align*}
  \begin{split}
 {\mathcal{L}}_\epsilon {\tilde{\phi}}
  & = \lambda \alpha_1 {\tilde{u}} - \lambda
 \alpha_2 {\tilde{v}} ,\\
   \frac{\partial {\tilde{u}}}{\partial t} + D_1 {\mathcal{L}}_\epsilon {\tilde{u}} &= - \frac{D_1}{g^2}\nabla \cdot
   \left(  \alpha_1 {\tilde{u}}   g^2 JJ^\tau \nabla (\phi -L_3^\epsilon)
   + \alpha_1 (u-L_1^\epsilon) g^2 JJ^\tau \nabla{\tilde{\phi}}
   \right ) ,\\
    \frac{\partial {\tilde{v}}}{\partial t} + D_2 {\mathcal{L}}_\epsilon {\tilde{v}} &= - \frac{D_2}{g^2}\nabla \cdot
   \left(  \alpha_2 {\tilde{v}}   g^2 JJ^\tau \nabla (\phi -L_3^\epsilon)
   + \alpha_2 (v-L_2^\epsilon) g^2 JJ^\tau \nabla{\tilde{\phi}}
   \right ) .
 \end{split}
 \end{align*}
 From the above system,  one derives
  \begin{align}
  \label{pflemma351}
  \begin{split}
 {\mathcal{L}}_\epsilon (t{\tilde{\phi}})
  & = \lambda \alpha_1 t {\tilde{u}} - \lambda
 \alpha_2  t {\tilde{v}} ,\\
   \frac{\partial  }{\partial t} (t{\tilde{u}} )  + D_1 {\mathcal{L}}_\epsilon (t {\tilde{u}})
    &={\tilde{u}} - \frac{D_1}{g^2}\nabla \cdot
   \left(  \alpha_1 (t {\tilde{u}} )  g^2 JJ^\tau \nabla (\phi -L_3^\epsilon)
   + \alpha_1 (u-L_1^\epsilon) g^2 JJ^\tau \nabla (t {\tilde{\phi}})
   \right ) ,\\
    \frac{\partial  }{\partial t} (t{\tilde{v}})
     + D_2 {\mathcal{L}}_\epsilon  (t {\tilde{v}}) &= {\tilde{v}}
     + \frac{D_2}{g^2}\nabla \cdot
   \left(  \alpha_2  (t {\tilde{v}})   g^2 JJ^\tau \nabla (\phi -L_3^\epsilon)
   + \alpha_2 (v-L_2^\epsilon) g^2 JJ^\tau \nabla (t {\tilde{\phi}})
   \right ) .
 \end{split}
 \end{align}
 The first equation in \eqref{pflemma351}  gives
  \begin{equation}
 \label{pflemma352}
 \| {\mathcal{L}}_\epsilon (t \tilde{\phi}) \|_{H_\epsilon}
 \le C \left (
  \| t {\tilde{u}} \|_{H_\epsilon} +
  \| t {\tilde{v}} \|_{H_\epsilon} \right ).
  \end{equation}
  Taking the inner product of the second equation
  in \eqref{pflemma351} with ${ t \tilde{u}}$ in $H_\epsilon$, we have
  \begin{align}
  \label{pflemma353}
  \begin{split}
  {\frac 12} {\frac d{dt}} \| {t \tilde{u}} \|^2_{H_\epsilon}
  + D_1 \| {\mathcal{L}}_\epsilon^{\frac 12} ({t \tilde{u}}) \|^2_{H_\epsilon}
  = & D_1 \alpha_1 \int {\frac {g^2}{\epsilon^2}} {t \tilde{u}}
  J^\tau \nabla (\phi- L_3^\epsilon) \cdot J^\tau \nabla ({t \tilde{u}})\\
  & + D_1 \alpha_1 \int {\frac {g^2}{\epsilon^2}} (u-L_1^\epsilon) J^\tau \nabla ({t \tilde{\phi}}) \cdot
  J^\tau \nabla ({t \tilde{u}})
  + t \| {\tilde{u}} \|^2_{H_\epsilon}.
  \end{split}
 \end{align}
  By \eqref{alltimeh1bd}, the first term on the right-hand side
 of \eqref{pflemma353} is bounded by
 \begin{align}
 \label{pflemma354}
 \begin{split}
  C \|{t \tilde{u}} \|_3 \| J^\tau \nabla (\phi -L_3^\epsilon) \|_6 \| J^\tau \nabla ({t \tilde{u}})   \|_2
 &  \le C  \|{t \tilde{u}} \|_2^{\frac 12}   \|{t \tilde{u}} \|^{\frac 12}_{H^1} \| J^\tau \nabla (\phi - L_3^\epsilon) \|_{H^1}\| J^\tau \nabla (t \tilde{u}) \|_2 \\
 & \le  C  \|{t \tilde{u}} \|_{H_\epsilon}^{\frac 12}
 \| {\mathcal{L}}^{\frac 12}_\epsilon ({t \tilde{u}}) \|^{\frac 32}_{H_\epsilon}
 \left (  \| {\mathcal{L}}_\epsilon \phi \|_{H_\epsilon}
  + \| {\mathcal{L}}_\epsilon L_3^\epsilon \|_{H_\epsilon}   \right )\\
    & \le  C  \|{t \tilde{u}} \|_{H_\epsilon}^{\frac 12}
    \| {\mathcal{L}}^{\frac 12}_\epsilon ({t \tilde{u}}) \|^{\frac 32}_{H_\epsilon}
 \le {\frac 18} D_1 \| {\mathcal{L}}^{\frac 12}_\epsilon ({t \tilde{u}}) \|^{2}_{H_\epsilon}
 + C \| t {\tilde{u}} \|^2_{H_\epsilon}.
   \end{split}
 \end{align}
 By \eqref{pflemma352}, the second term on the right-hand side of \eqref{pflemma353} is less than
 \begin{align}
 \label{pflemma355}
 \begin{split}
 C \| u -L_1^\epsilon \|_\infty \| J^\tau \nabla (t \tilde{\phi}) \|_2   \| J^\tau \nabla (t \tilde{u}) \|_2
 & \le {\frac 18} D_1 \| {\mathcal{L}}_\epsilon^{\frac 12} (t \tilde{u}) \|^2_{H_\epsilon}
 + C \| {\mathcal{L}}_\epsilon^{\frac 12}  (t \tilde{\phi})  \|^2_{H_\epsilon}\\
  & \le {\frac 18} D_1 \| {\mathcal{L}}_\epsilon^{\frac 12} (t \tilde{u}) \|^2_{H_\epsilon}
 + C \left ( \| t \tilde{u} \|^2_{H_\epsilon} +  \| t \tilde{v} \|^2_{H_\epsilon}  \right ).
 \end{split}
 \end{align}
 Multiplying the second equation in \eqref{absPNP} by $t\tilde{u}$, after simple computations, we find
 that the last term on the right-hand side of \eqref{pflemma353}
 satisfies
 \begin{align}
 \label{pflemma356}
 \begin{split}
 t \| \tilde{u} \|^2_{H_\epsilon} \le &
  C \| {\mathcal{L}}_\epsilon^{\frac 12} (t \tilde{u}) \|_{H_\epsilon}
  \left ( \| {\mathcal{L}}_\epsilon^{\frac 12} u \|_{H_\epsilon} +
   \| {\mathcal{L}}_\epsilon^{\frac 12} L_1^\epsilon \|_{H_\epsilon}  +
   \| {\mathcal{L}}_\epsilon^{\frac 12} \phi \|_{H_\epsilon}
  + \| {\mathcal{L}}_\epsilon^{\frac 12} L_3^\epsilon \|_{H_\epsilon}    \right )\\
   \le & {\frac 18} D_1 \| {\mathcal{L}}_\epsilon^{\frac 12} (t \tilde{u}) \|^2_{H_\epsilon}+ C.
 \end{split}
    \end{align}
 Combining the estimates   in \eqref{pflemma353}-\eqref{pflemma356}, we get
 \begin{equation}\label{pflemma357}
    {\frac d{dt}} \| t {\tilde{u}} \|^2_{H_\epsilon} + D_1 \| {\mathcal{L}}_\epsilon^{\frac 12} (t \tilde{u}) \|^2_{H_\epsilon}\le C   \left ( \| t \tilde{u} \|^2_{H_\epsilon} +
 \| t \tilde{v} \|^2_{H_\epsilon}  \right ) +C.
 \end{equation}
 Similarly,
 \begin{equation}
    \label{pflemma358}
    {\frac d{dt}} \| t {\tilde{v}} \|^2_{H_\epsilon} + D_1 \| {\mathcal{L}}_\epsilon^{\frac 12} (t \tilde{v}) \|^2_{H_\epsilon} \le C   \left ( \| t \tilde{u} \|^2_{H_\epsilon} +
 \| t \tilde{v} \|^2_{H_\epsilon}  \right ) +C.
 \end{equation}
 Finally, from \eqref{pflemma357}-\eqref{pflemma358}, we have
 \[ {\frac d{dt}}  \left (\| t {\tilde{u}} \|^2_{H_\epsilon} +\| t {\tilde{v}} \|^2_{H_\epsilon}  \right )
  + C_1 \left ( \| {\mathcal{L}}_\epsilon^{\frac 12} (t \tilde{u}) \|^2_{H_\epsilon}
  +   \| {\mathcal{L}}_\epsilon^{\frac 12} (t \tilde{v}) \|^2_{H_\epsilon}\right )
    \le C   \left (\| t \tilde{u} \|^2_{H_\epsilon} + \| t \tilde{v} \|^2_{H_\epsilon}  \right ) +C,\]
 which, along with   Gronwall's lemma, concludes the proof.
 \end{proof}

We now describe the analogue of Lemma \ref{lemma35} for the one-dimensional
limiting system \eqref{limPNP}-\eqref{limIC}.
\begin{lem}\label{lemma36}
Given $R>0$, there exists $K$ depending
only on $R$ such that for any  $(u_0, v_0) \in \Sigma_0$
with $\|(u_0, v_0) \|_{H^1 \times H^1} \le R$, the
solution $(u,v)$ of  problem \eqref{limPNP}-\eqref{limIC} satisfies
\[t^2 \left ( \|    {\frac {\partial \phi}{\partial t}} \|^2_{H^2}
 +\|{\frac {\partial u}{\partial t}} \|^2_{2} + \|{\frac {\partial v}{\partial t}} \|^2_{2} \right )
+ \int_0^t s^2 \left (\|{\frac {\partial u}{\partial s}} \|^2_{H^1} + \|{\frac {\partial v}{\partial s}} \|^2_{H^1}\right ) ds\le Ke^{Kt}, \quad t \ge 0.\]
\end{lem}
\begin{proof}
The proof is similar to that of Lemma \ref{lemma35} but simpler, and therefore
omitted here. \end{proof}

\subsection{Upper Semicontinuity }
In this section, we establish the upper semicontinuity of global
attractors ${\mathcal{A}}_\epsilon$ at $\epsilon =0$. We first
compare the solutions of the three-dimensional  problem
\eqref{PNP}-\eqref{IC}  and  the one-dimensional limiting problem
\eqref{limPNP}-\eqref{limIC}, and then establish the relationships
between the global attractors of the two dynamical systems.

In what follows, we reformulate   limiting  system \eqref{limPNP}
as an operator equation. Let $H_0$ be the $L^2(0,1)$ space with
  the inner product $(\cdot,\cdot)_{H_0}$ given by
\[(u,v)_{H_0}=\int_0^1 g_0^2uv\,dx,\]
  and  let $a_0(\cdot,\cdot)$  be the   bilinear form on $\left(H^1_0
  (0,1)\right)^2$:
 \[a_0(w_1,w_2)= \left ({\frac {dw_1}{dx}}, {\frac {dw_2}{dx}} \right )_{H_0}
 =\int_0^1 g_0^2 \, {\frac {dw_1}{dx}} \, {\frac {dw_2}{dx}}  \,dx.\]
For $f \in L^2(\Omega)$, let $M(f) \in L^2(0,1)$ be the function:
\[(M(f))(x)= {\frac 1\pi} \int_{\bbD} f(x,y,z)\,dy\,dz.\]

\begin{lem}
\label{meanlemma} Suppose $f \in H^1(\Omega)$. Then we have
\[\|f - M(f) \|_{H_\epsilon} \le C \epsilon \| f\|_{X_\epsilon}.\]
\end{lem}
\begin{proof}   Notice that
  \begin{align}\label{pfmean1}
 \|f - M(f) \|^2_2=& \int_{x=0}^1 \int_{\bbD}
 \left | f(x,y,z) - {\frac 1\pi}\int_{\bbD}  f(x, u,v) du dv \right |^2 dy dz dx.
  \end{align}
 Using the identity
  \begin{align*}
 f(x, r\cos\theta, r \sin \theta ) = &f(x, \rho \cos \phi, \rho \sin \phi ) - \int_{\theta}^\phi
 {\frac {\partial}{\partial t}}f(x, \rho \cos t, \rho \sin t ) dt\\
  &-     \int_r^{\rho} {\frac {\partial}{\partial \tau}} f(x, \tau \cos\theta, \tau \sin \theta ) d\tau,
 \end{align*}
one can write
\begin{align*}
 & {\frac 1\pi} \int_{r=0}^1 \int_{\theta=0}^{2 \pi} f(x, r\cos
 \theta, r \sin \theta ) r dr d \theta = f(x, \rho \cos \phi, \rho \sin \phi )\\
 & \qquad -{\frac 1\pi} \int_{r=0}^1 \int_{\theta=0}^{2 \pi}
 \left (\int_{t = \theta}^\phi \left (- \rho \sin t {\frac {\partial f}{\partial y}}(x, \rho \cos t, \rho \sin t )
 + \rho \cos t {\frac {\partial f}{\partial z}}(x, \rho \cos t, \rho \sin t )\right )dt \right ) r dr d \theta \\
 & \qquad -{\frac 1\pi} \int_{r=0}^1 \int_{\theta=0}^{2 \pi}
 \left ( \int_{\tau = r}^\rho \left (  \cos \theta{\frac {\partial f}{\partial y}}(x, \tau \cos\theta, \tau \sin \theta ) + \sin \theta{\frac {\partial f}{\partial z}}(x, \tau \cos \theta, \tau \sin \theta )\right )d\tau \right) r dr d \theta
  \end{align*}
  Then, after simple computations,
   Lemma \ref{meanlemma} follows from \eqref{pfmean1} and the above.
\end{proof}

 Let $(\psi,P,Q)\in ( {H}_D^1(\Omega))^3$ and let
$(\phi_{\epsilon},u_{\epsilon},v_{\epsilon})$ be a solution of system~\eqref{PNP}.   In view of
the boundary condition~\eqref{BV} and the choices of
$L_k^\epsilon$ for $k=1,2,3$, we have
\begin{align}\label{VP}\begin{split}
-a_{\epsilon}(\phi_{\epsilon}-L^\epsilon_3,\psi)=&-\lambda
\alpha_1(u_{\epsilon}-L_1^\epsilon,\psi)_{H_{\epsilon}}+
\lambda \alpha_2(v_{\epsilon}-L_2^\epsilon,\psi)_{H_{\epsilon}},\\
\frac{1}{D_1}\left(\frac{\partial u_{\epsilon}}{\partial t},
P\right)_{H_{\epsilon}}=&
-a_{\epsilon}(u_{\epsilon}-L_1^\epsilon,P)+\alpha_1\left((u_{\epsilon}-L_1^\epsilon){\mathcal
L}^{1/2}_{\epsilon}
(\phi_{\epsilon}-L_3^\epsilon), {\mathcal L}^{1/2}_{\epsilon}P\right)_{H_{\epsilon}},\\
\frac{1}{D_2}\left(\frac{\partial v_{\epsilon}}{\partial t},Q\right)_{H_{\epsilon}}
=&-a_{\epsilon}(v_{\epsilon}-L_2^\epsilon,Q)
  -\alpha_2\left((v_{\epsilon}-L_2^\epsilon){\mathcal L}^{1/2}_{\epsilon} (\phi_{\epsilon}-L_3^\epsilon), {\mathcal L}^{1/2}_{\epsilon}Q\right)_{H_{\epsilon}}.
\end{split}
\end{align}
Let $(\phi, u, v)$ be the solution of the limiting system~\eqref{limPNP}.
View $(\phi, u, v)$ as an element in $({H}^1_D(\Omega))^3$.
Then a  direct computation yields that, for
  $(\psi,P,Q)\in ({H}^1_D(\Omega))^3$,
  \begin{align}
\label{one2three}
\begin{split}
-a_{\epsilon}(\phi-L^0_3,\psi)=&-\lambda\alpha_1(u-L^0_1,\psi)_{H_{\epsilon}}
+\lambda\alpha_2(v-L^0_2,\psi)_{H_{\epsilon}}+F(\phi-L_3^0,\psi),\\
\frac{1}{D_1}\left(\frac{\partial u}{\partial t},P\right)_{H_{\epsilon}}
=&- a_{\epsilon}(u-L^0_1,P)+\alpha_1\left((u-L^0_1){\mathcal L}^{1/2}_{\epsilon}(\phi-L^0_3),
 {\mathcal L}^{1/2}_{\epsilon} P\right)_{H_{\epsilon}}\\
 & +G_1(u-L_1^0,\phi-L_3^0,P),\\
\frac{1}{D_2}\left(\frac{\partial v}{\partial t},Q\right)_{H_{\epsilon}}
=&- a_{\epsilon}(v-L^0_2,Q)-\alpha_2\left((v-L^0_2){\mathcal L}^{1/2}_{\epsilon}(\phi-L^0_3),
{\mathcal L}^{1/2}_{\epsilon} Q\right)_{H_{\epsilon}}\\
&+G_2(v-L_2^0,\phi-L_3^0,Q),
\end{split}
\end{align}
where, for appropriate functions $p$, $q$ and $r$, and for $i=1,2$,
\begin{align}\label{fg}\begin{split}
F(p, q)=& \left(\left(\frac{\partial_x g^2}{g^2}
-\frac{\partial_x
g_0^2}{g_0^2}\right) p_x,q\right)_{H_{\epsilon}}
  +\left(\frac{g_x}{g} p_x,yq_y+zq_z\right)_{H_{\epsilon}},\\
G_i(p, q, r)=&  -\left(\left(\frac{\partial_x g^2}{g^2}
  -\frac{\partial_x g_0^2}{g_0^2}\right)p_x,r\right)_{H_{\epsilon}}
  -\left(\frac{g_x}{g} p_x,yr_y+zr_z\right)_{H_{\epsilon}}\\
&+(-1)^{i+1}\alpha_i\left(\left(\frac{\partial_x g^2}{g^2}
  -\frac{\partial_x g_0^2}{g_0^2}\right) p
    q_x,r\right)_{H_{\epsilon}}+(-1)^{i+1}\alpha_i\left(\frac{g_x}{g} p
 q_x,yr_y+zr_z\right)_{H_{\epsilon}}.
 \end{split}
\end{align}

 Let
\begin{equation}
\label{diffvar}
\psi^{\epsilon}=\phi_{\epsilon}-L_3^\epsilon-(\phi-L^0_3), \;
P^{\epsilon}=u_{\epsilon}-L_1^\epsilon-(u -L^0_1),\;
Q^{\epsilon}=v_{\epsilon}-L_2^\epsilon-(v -L^0_2).
\end{equation}

 Upon subtracting \eqref{one2three} from \eqref{VP}, we obtain that
for any $(\psi, P, Q) \in H^1_D(\Omega)^3$,
\begin{align}
  a_{\epsilon}(\psi^{\epsilon},\psi)
  =& \lambda\alpha_1(P^{\epsilon},\psi)_{H_{\epsilon}}
  - \lambda\alpha_2 (Q^{\epsilon},\psi)_{H_{\epsilon}}+F(\phi-L_3^0,\psi), \label{diffeq1} \\
   \frac{1}{D_1}\left(\partial_t P^{\epsilon}, P \right)_{H_{\epsilon}}
 =& -a_{\epsilon}( P^{\epsilon}, P )+
 \alpha_1\left((u_{\epsilon}-L_1^\epsilon) {\mathcal{L}}^{\frac 12}_{\epsilon}  \psi^{\epsilon},
 {\mathcal{L}}^{\frac 12}_{\epsilon} P \right)_{H_{\epsilon}}\nonumber\\
&+\alpha_1\left(P^{\epsilon}
 {\mathcal{L}}^{\frac 12}_{\epsilon} (\phi-L^0_3),
  {\mathcal{L}}^{\frac 12}_{\epsilon} P
   \right)_{H_{\epsilon}}-G_1(u-L_1^0,\phi-L_3^0,P),\label{diffeq2} \\
\frac{1}{D_2}\left(\partial_t Q^{\epsilon},Q \right)_{H_{\epsilon}}
=& -a_{\epsilon}(Q^{\epsilon},Q ) -\alpha_2\left((v_{\epsilon}-L_2^\epsilon)
{\mathcal{L}}^{\frac 12}_{\epsilon}  \psi^{\epsilon},
 {\mathcal{L}}^{\frac 12}_{\epsilon} Q
\right)_{H_{\epsilon}}\nonumber\\
&-\alpha_2\left(Q^{\epsilon}
 {\mathcal{L}}^{\frac 12}_{\epsilon} (\phi-L^0_3),   {\mathcal{L}}^{\frac 12}_{\epsilon} Q
\right)_{H_{\epsilon}}-G_2(v-L_2^0,\phi-L_3^0,Q).\label{diffeq3}
\end{align}

 For the above system, we    have the following estimates.
\begin{lem}
\label{lemma42} There exists    $\epsilon_1 >0 $ such that, for
any $R>0$, there exists a constant $K$ depending on $R$ such that,
for any $0<\epsilon \le \epsilon_1$ and  $(u_0,v_0)\in
\Sigma_{\epsilon}$ with $\|(u_0, v_0)\|_{X_\epsilon \times
X_\epsilon} \le R$, the following holds:
 \[\|P^{\epsilon}(t)\|^2_{H_{\epsilon}}+\|Q^{\epsilon}(t)\|^2_{H_{\epsilon}}
 + \| \psi^{\epsilon} (t) \|^2_{X_\epsilon}  + \int_0^t\left(\|P^{\epsilon}(s)\|^2_{X_{\epsilon}}
   +\|Q^{\epsilon}(s)\|^2_{X_{\epsilon}}
  \right)ds
 \le  \epsilon K e^{Kt}, \quad t \ge 0,  \]
where $(\phi_\epsilon,u_{\epsilon}, v_{\epsilon})$ is the solution of problem
\eqref{PNP}-\eqref{IC} with  the initial condition $(u_0,v_0)$,
$(\phi, u, v)$ is the solution of problem \eqref{limPNP}-\eqref{limIC}
with  the initial condition $( M(u_0), M(v_0))$, and $(\psi^\epsilon, P^\epsilon, Q^\epsilon)$ is given by \eqref{diffvar}.
\end{lem}

\begin{proof}
It  follows from   \eqref{diffeq1}--\eqref{diffeq3} that
\begin{align}
  a_{\epsilon}(\psi^{\epsilon},\psi^{\epsilon})
  =& \lambda\alpha_1(P^{\epsilon},\psi^{\epsilon})_{H_{\epsilon}}
  - \lambda\alpha_2 (Q^{\epsilon},\psi^{\epsilon})_{H_{\epsilon}}+F(\phi-L_3^0,\psi^{\epsilon}), \label{pfprop411}\\
 \frac{1}{D_1}\left(\partial_t P^{\epsilon},
P^{\epsilon}\right)_{H_{\epsilon}} =& -a_{\epsilon}( P^{\epsilon},
P^{\epsilon})+
 \alpha_1\left((u_{\epsilon}-L_1^\epsilon) {\mathcal{L}}^{\frac 12}_{\epsilon}  \psi^{\epsilon},
 {\mathcal{L}}^{\frac 12}_{\epsilon} P^{\epsilon}\right)_{H_{\epsilon}}\nonumber\\
&+\alpha_1\left(P^{\epsilon}
 {\mathcal{L}}^{\frac 12}_{\epsilon} (\phi-L^0_3),
  {\mathcal{L}}^{\frac 12}_{\epsilon} P^\epsilon
   \right)_{H_{\epsilon}} -G_1(u-L_1^0,\phi-L_3^0,P^{\epsilon}),\label{pfprop412}\\
\frac{1}{D_2}\left(\partial_t Q^{\epsilon},
Q^{\epsilon}\right)_{H_{\epsilon}} =& -a_{\epsilon}(Q^{\epsilon},
Q^{\epsilon}) -\alpha_2\left((v_{\epsilon}-L_2^\epsilon)
{\mathcal{L}}^{\frac 12}_{\epsilon}  \psi^{\epsilon},
 {\mathcal{L}}^{\frac 12}_{\epsilon} Q^{\epsilon}
\right)_{H_{\epsilon}}\nonumber\\
&-\alpha_2\left(Q^{\epsilon}
 {\mathcal{L}}^{\frac 12}_{\epsilon} (\phi-L^0_3),   {\mathcal{L}}^{\frac 12}_{\epsilon} Q^\epsilon
\right)_{H_{\epsilon}}-G_2(v-L_2^0,\phi-L_3^0,Q^{\epsilon}).\label{pfprop413}
\end{align}
Next, we estimate  each term on the right-hand sides of
\eqref{pfprop411}-\eqref{pfprop413}. The first two terms on the
right-hand side of \eqref{pfprop411} are bounded by:
\begin{align}\label{pfprop414}
 \lambda \alpha_1  \left|(P^{\epsilon},\psi^{\epsilon})_{H_{\epsilon}}\right|+
\lambda \alpha_2 \left|(Q^{\epsilon},\psi^{\epsilon})_{H_{\epsilon}}\right|
 \le & C
 \left (\|P^{\epsilon}\|_{H_{\epsilon}}+\|Q^{\epsilon}\|_{H_{\epsilon}} \right )
  \|\psi^{\epsilon}\|_{H_{\epsilon}}\nonumber\\
 \le & C \left (\|P^{\epsilon}\|^2_{H_{\epsilon}}+\|Q^{\epsilon}\|^2_{H_{\epsilon}}
 \right )+ {\frac 14}  a_{\epsilon}(\psi^{\epsilon},\psi^{\epsilon}).
 \end{align}
 By \eqref{g_inequality}  we find that $g$ satisfies
\[\left|\frac{\partial_x g^2}{g^2} -\frac{\partial_x g_0^2}{g_0^2}\right|\le C\epsilon.\]
Then, by Lemma \ref{lemma34}, the first term in $F(\phi-L_3^0,\psi^{\epsilon})$  on the right-hand side of
  \eqref{pfprop411} is less than
\begin{equation}
\label{pfprop415}
 \left|\left(\left(\frac{\partial_x g^2}{g^2}
  -\frac{\partial_x g_0^2}{g_0^2}\right)
   (\phi-L^0_3)_x,\psi^{\epsilon}\right)_{H_{\epsilon}}\right|
   \le C \epsilon    \|\psi^{\epsilon}\|^2_{H_{\epsilon}}
\le C\epsilon^2 + {\frac 14}   a_{\epsilon}(\psi^{\epsilon},\psi^{\epsilon}).
 \end{equation}
 It follows from \eqref{alltimeh1bd} and  Lemma  \ref{lemma34} that
 the  second   term in $F(\phi-L_3^0,\psi^{\epsilon})$ on the right-hand side of \eqref{pfprop411} is
 bounded by, for $t \ge 0$,
 \begin{align}
 \label{pfprop416}
  &\left|\left(\frac{g_x}{g}(\phi-L^0_3)_x, y\partial_y(\phi_{\epsilon}-L_3^\epsilon)+
z\partial_z(\phi_{\epsilon}-L_3^\epsilon)\right)_{H_{\epsilon}}\right| \nonumber\\
&  \le C  \left ( \|\partial_y \phi_{\epsilon} \|_{H_\epsilon}  +
\|\partial_z \phi_{\epsilon} \|_{H_\epsilon}
 +\|\partial_y L_3^\epsilon \|_{H_\epsilon}  + \|\partial_z L_3^\epsilon \|_{H_\epsilon} \right )
\le  C \epsilon .
\end{align}
By \eqref{pfprop411} and \eqref{pfprop414}-\eqref{pfprop416}, we
obtain, for all $t \ge 0$,
\begin{equation}
\label{pfprop417}
 \|\psi^{\epsilon} (t) \|^2_{X_{\epsilon}}\le C  a_{\epsilon}(\psi^{\epsilon},\psi^{\epsilon})
 \le C
 \left (\|P^{\epsilon}\|^2_{H_{\epsilon}}+\|Q^{\epsilon}\|^2_{H_{\epsilon}}
 \right )
+ C \epsilon.
\end{equation}
We now deal with the right-hand side of \eqref{pfprop412}. By
\eqref{pfprop417}, the second term on the right-hand side of
\eqref{pfprop412} is less than
 \begin{align}
 \label{pfprop418}
 \alpha_1 \left|\left((u_{\epsilon}-L_1){\mathcal{L}}_\epsilon^{\frac 12}
 \psi^{\epsilon}, {\mathcal{L}}_\epsilon^{\frac 12}
P^{\epsilon}\right)_{H_{\epsilon}}\right|
 \le   C(\|\psi^{\epsilon}\|^2_{X_{\epsilon}}
+ \|P^{\epsilon}\|^2_{X_{\epsilon}})
 \le
   C(\|P^{\epsilon}\|^2_{X_{\epsilon}}
+ \|Q^{\epsilon}\|^2_{X_{\epsilon}}) + C\epsilon.
\end{align}
 Since  the functions $\phi$ and $L_3^0$ depend on $x \in
(0,1)$ only,    we have
\[ {\mathcal{L}}^{\frac 12}_\epsilon \phi = J^\tau \nabla \phi
= (\partial_x \phi, 0, 0)^\tau, \qquad
 {\mathcal{L}}^{\frac 12}_\epsilon L_3^0  = J^\tau \nabla L_3^0
= (\partial_x L_3^0, 0, 0)^\tau, \]
 which, along with Lemma \ref{lemma34} and the first equation of  \eqref{limPNP},
 implies that, for all $t \ge 0$,
\begin{equation}
\label{pfprop419}
 \|{\mathcal{L}}^{\frac 12}_\epsilon \phi \|_{\infty}
 = \| \partial_x \phi \|_\infty
 \le  C \| \partial_x \phi \|_{H^1 }
 \le C \|\phi \|_{H^2}
 \le C.
 \end{equation}
 By \eqref{pfprop419}, the third term on the right-hand side
 of \eqref{pfprop412} is bounded by
\begin{eqnarray}
\label{pfprop4110}
\alpha_1 \left | \left(P^{\epsilon}
 {\mathcal{L}}^{\frac 12}_{\epsilon} (\phi-L^0_3),
  {\mathcal{L}}^{\frac 12}_{\epsilon} P^\epsilon
  \right)_{H_{\epsilon}} \right |
  &\le& \alpha_1 \left ( \|{\mathcal{L}}^{\frac 12}_{\epsilon} \phi
  \|_\infty +
  \|{\mathcal{L}}^{\frac 12}_{\epsilon} L^0_3
  \|_\infty \right ) \|P^{\epsilon}\|_{H_\epsilon} \| {\mathcal{L}}^{\frac 12}_{\epsilon} P^\epsilon
  \|_{H_\epsilon} \nonumber \\
  &\le  &
  C \|P^{\epsilon}\|_{H_\epsilon}^2  + {\frac 14}a_{\epsilon}(P^{\epsilon}, P^{\epsilon}) .
 \end{eqnarray}
 Note that  the term $G_1$ on the right-hand side of \eqref{pfprop412}
 can be estimated in a similar manner as
 \eqref{pfprop414}-\eqref{pfprop416}.
 Therefore, it follows from
  \eqref{pfprop412} and \eqref{pfprop417}-\eqref{pfprop4110} that,
  for $t \ge 0$,
\begin{align}
\label{pfprop4111}
  \frac{d}{dt}\|P^{\epsilon}\|^2_{H_{\epsilon}}
    +  \| P^{\epsilon}\|_{X_\epsilon}^2 \le
     C (\|P^{\epsilon}\|^2_{H_{\epsilon}}
+ \|Q^{\epsilon}\|^2_{H_{\epsilon}}) +C \epsilon.
\end{align}
Similarly,   $Q^\epsilon$ satisfies, for $t \ge 0$,
\begin{align}
\label{pfprop4112}
  \frac{d}{dt}\|Q^{\epsilon}\|^2_{H_{\epsilon}}
    +  \| Q^{\epsilon}\|_{X_\epsilon}^2 \le
     C (\|P^{\epsilon}\|^2_{H_{\epsilon}}
+ \|Q^{\epsilon}\|^2_{H_{\epsilon}}) +C \epsilon.
 \end{align}
 Then, it follows from \eqref{pfprop4111}-\eqref{pfprop4112}
 that, for $t \ge 0$,
\begin{align}
\label{pfprop4113}
  \frac{d}{dt} \left (\|P^{\epsilon}\|^2_{H_{\epsilon}}
  + \|Q^{\epsilon}\|^2_{H_{\epsilon}} \right )
  + \| P^{\epsilon}\|_{X_\epsilon}^2
    +  \| Q^{\epsilon}\|_{X_\epsilon}^2 \le
     C (\|P^{\epsilon}\|^2_{H_{\epsilon}}
+ \|Q^{\epsilon}\|^2_{H_{\epsilon}}) +C \epsilon.
 \end{align}
 By Gronwall's lemma, we get
\begin{align}
\label{pfprop4114}
   & \|P^{\epsilon} (t) \|^2_{H_{\epsilon}}
  + \|Q^{\epsilon} (t) \|^2_{H_{\epsilon}}
   \le e^{Ct} \left (
 \|P^{\epsilon} (0) \|^2_{H_{\epsilon}}
  + \|Q^{\epsilon} (0) \|^2_{H_{\epsilon}}
\right )
 + \epsilon e^{C t} \nonumber \\
   &\le C e^{Ct} \left (
 \|u_0 - M(u_0) \|^2_{H_{\epsilon}}
 + \|L_1^\epsilon - L_1^0 \|^2_{H_{\epsilon}}
 + \|v_0 - M(v_0) \|^2_{H_{\epsilon}}
 + \|L_2^\epsilon - L_2^0 \|^2_{H_{\epsilon}}
\right )
  + \epsilon e^{C t}.
 \end{align}
By \eqref{AF} we see that   $L_1, L_2 \in W^{1,
\infty}(\Omega_\epsilon)$,  and hence, for $k=1,2$,
\begin{equation}
\label{pfprop4115}
 \|L_k^\epsilon - L_k^0 \|^2_{H_{\epsilon}}
 = \| \int_0^1  \left ( y g  {\frac {\partial L_k}{\partial Y}}(x, s  g y, s g z )
 + z g  {\frac {\partial L_k}{\partial  Z}}(x, s g y, s g z ) \right ) ds \|^2_{H_{\epsilon}}
 \le C\epsilon^2.
\end{equation}
From \eqref{pfprop4114}-\eqref{pfprop4115} and Lemma
\ref{meanlemma}, we find that
\begin{equation}
\label{pfprop4116}
 \|P^{\epsilon} (t) \|^2_{H_{\epsilon}}
  + \|Q^{\epsilon} (t) \|^2_{H_{\epsilon}}
   \le \epsilon (C+1) e^{Ct}.
   \end{equation}
Integrating \eqref{pfprop4113} between $0$ and $t$, by
\eqref{pfprop4116} we conclude Lemma \ref{lemma42}.
\end{proof}

Next, we improve the uniform estimates in $\epsilon$  given in
Lemma \ref{lemma42}.
\begin{lem}
\label{lemma43}
There exists    $\epsilon_1 >0 $ such that, for
any $R>0$, there exists a constant $K$ depending on $R$ such that,
for any $0<\epsilon \le \epsilon_1$ and  $(u_0,v_0)\in
\Sigma_{\epsilon}$ with $\|(u_0, v_0)\|_{X_\epsilon \times
X_\epsilon} \le R$, the following holds:
$$   t^2 \left ( \|   {\frac {\partial P^\epsilon}{\partial t}} \|_{H_\epsilon}^2
 +  \|   {\frac {\partial Q^\epsilon}{\partial t}} \|_{H_\epsilon}^2  \right )
    +  t  \left (  \|   P^\epsilon \|^2_{X_\epsilon}
    +    \|   Q^\epsilon \|^2_{X_\epsilon}  \right )
    \le \sqrt{\epsilon} Ke^{Kt},  \quad t \ge 0,
    $$
  where $(\psi^\epsilon, P^\epsilon, Q^\epsilon)$ is given by \eqref{diffvar},
   $(\phi_\epsilon,
u_{\epsilon}, v_{\epsilon})$ is the solution of problem
\eqref{PNP}-\eqref{IC} with  the initial condition $(u_0,v_0)$,  and
$(\phi, u, v)$ is the solution of problem \eqref{limPNP}-\eqref{limIC}
with  the initial condition $( M(u_0), M(v_0)).$
\end{lem}

\begin{proof}  Denote by
\begin{equation}
\label{pflemma431}
{\tilde{P}}^\epsilon = {\frac {\partial P^\epsilon}{\partial t}}, \quad
{\tilde{Q}}^\epsilon = {\frac {\partial Q^\epsilon}{\partial t}}, \quad
{\tilde{\psi}}^\epsilon = {\frac {\partial \psi^\epsilon}{\partial t}}.
\end{equation}
Differentiating systems  \eqref{diffeq1}--\eqref{diffeq3}   with respect to $t$, multiplying the resulting systems by $t$, replacing
  $ \psi$, $P$ and $Q$ by $t \tilde \psi$,
 $t\tilde P$ and $t \tilde Q$,  respectively,  we obtain
 \begin{align} \label{pflemma432}
  a_{\epsilon}(t {\tilde{\psi}}^{\epsilon}, t {\tilde{\psi}}^{\epsilon})
  =& \lambda\alpha_1(t \tilde P^{\epsilon},  t {\tilde{\psi}}^{\epsilon})_{H_{\epsilon}}
  - \lambda\alpha_2 (t \tilde Q^{\epsilon},  t {\tilde{\psi}}^{\epsilon})_{H_{\epsilon}}
   + t  F(\phi_t,t  {\tilde{\psi}}^{\epsilon}),
\end{align}
 \begin{align} \label{pflemma433}
  {\frac {1}{2D_1}} {\frac d{dt}}  \| t   \tilde P^{\epsilon} \|^2_{H_{\epsilon}}
  &+a_{\epsilon}(  t \tilde P^{\epsilon},t \tilde P^{\epsilon})=
 \alpha_1t \left( \partial_t u_{\epsilon}  {\mathcal{L}}^{\frac 12}_{\epsilon}  \psi^{\epsilon},
 {\mathcal{L}}^{\frac 12}_{\epsilon}  t \tilde P^{\epsilon}  \right)_{H_{\epsilon}}
 + \alpha_1 t  \left ( (u_\epsilon -L_1^\epsilon) {\mathcal{L}}^{\frac 12}_\epsilon \tilde{\psi}^\epsilon,
 {\mathcal{L}}^{\frac 12}_\epsilon t \tilde P^{\epsilon} \right )_{H_\epsilon} \nonumber\\
 &+\alpha_1 t \left( \tilde P^{\epsilon}
 {\mathcal{L}}^{\frac 12}_{\epsilon} (\phi-L^0_3),
  {\mathcal{L}}^{\frac 12}_{\epsilon} t \tilde P^{\epsilon}
   \right)_{H_{\epsilon}}
   +\alpha_1 t \left(  P^{\epsilon}
 {\mathcal{L}}^{\frac 12}_{\epsilon} \phi_t,
  {\mathcal{L}}^{\frac 12}_{\epsilon} t \tilde P^{\epsilon}
   \right)_{H_{\epsilon}}
   \nonumber\\
   &-\alpha_1 t \left(\left(\frac{\partial_x g^2}{g^2}
  -\frac{\partial_x g_0^2}{g_0^2}\right)
   (u-L_1^0) \phi_{tx}, t \tilde P^{\epsilon} \right)_{H_{\epsilon}}\\
  &-\alpha_1 t  \left(\frac{g_x}{g} (u-L_1^0) \phi_{tx},
   t y\partial_y  \tilde P^{\epsilon}+
t z\partial_z   \tilde P^{\epsilon}  \right)_{H_{\epsilon}}\nonumber\\
&- t G_1(u_t,\phi-L_3^0, t \tilde P^{\epsilon})
 + {\frac 1{D_1}} ( \tilde P^{\epsilon} , t  \tilde P^{\epsilon} )_{H_\epsilon},\nonumber
 \end{align}
  \begin{align} \label{pflemma434}
 \frac{1}{2D_2} {\frac d{dt}}  \| t  \tilde Q^{\epsilon} \|^2 _{H_{\epsilon}}
 & + a_{\epsilon}(  t \tilde Q^{\epsilon}, t \tilde Q^{\epsilon}) = -
 \alpha_2 t \left( \partial_t v_{\epsilon}  {\mathcal{L}}^{\frac 12}_{\epsilon}  \psi^{\epsilon},
 {\mathcal{L}}^{\frac 12}_{\epsilon} t \tilde Q^{\epsilon} \right)_{H_{\epsilon}}
 - \alpha_2 t \left ( (v_\epsilon -L_2^\epsilon) {\mathcal{L}}^{\frac 12}_\epsilon \tilde{\psi}^\epsilon,
 {\mathcal{L}}^{\frac 12}_\epsilon   t \tilde Q^{\epsilon} \right )_{H_\epsilon} \nonumber\\
 &-\alpha_2 t  \left( \tilde Q^{\epsilon}
 {\mathcal{L}}^{\frac 12}_{\epsilon} (\phi-L^0_3),
  {\mathcal{L}}^{\frac 12}_{\epsilon}  t \tilde Q^{\epsilon}
   \right)_{H_{\epsilon}}
   -\alpha_2 t  \left(  Q^{\epsilon}
 {\mathcal{L}}^{\frac 12}_{\epsilon} \phi_t,
  {\mathcal{L}}^{\frac 12}_{\epsilon}  t \tilde Q^{\epsilon}
   \right)_{H_{\epsilon}}
   \nonumber\\
   &+\alpha_2 t  \left(\left(\frac{\partial_x g^2}{g^2}
  -\frac{\partial_x g_0^2}{g_0^2}\right)
   (v-L_2^0) \phi_{tx},  t \tilde Q^{\epsilon} \right)_{H_{\epsilon}}\\
 & + \alpha_2  t \left(\frac{g_x}{g} (v-L_2^0) \phi_{tx},
 t y\partial_y  \tilde Q^{\epsilon} +
 t z\partial_z  \tilde Q^{\epsilon}
  \right)_{H_{\epsilon}}\nonumber\\
 &- t G_2(v_t,\phi-L_3^0, t \tilde Q^{\epsilon}) + {\frac 1{D_2}} (\tilde Q^{\epsilon}, t\tilde Q^{\epsilon})_{H_\epsilon}.\nonumber
 \end{align}
 We now estimate every term involved in the above system. Note that
 \eqref{pflemma432} implies that
\begin{equation}
\label{pflemma435}
\| {\mathcal{L}}_\epsilon^{\frac 12} t \tilde \psi^\epsilon \|^2_{H_\epsilon}
 \le C \left (
  \| t \tilde P^\epsilon \|^2_{H_\epsilon}
  +
  \| t \tilde Q^\epsilon \|^2_{H_\epsilon}
   \right ) + C \epsilon^2 t^2.
   \end{equation}
By \eqref{alltimeh1bd} and Lemma \ref{lemma42}, we see that
the first term on the right-hand side of \eqref{pflemma433} is bounded by
\begin{align}
\label{pflemma436}
 & |\alpha_1t \left( \partial_t u_{\epsilon}  {\mathcal{L}}^{\frac 12}_{\epsilon}  \psi^{\epsilon},
 {\mathcal{L}}^{\frac 12}_{\epsilon}  t \tilde P^{\epsilon}  \right)_{H_{\epsilon}} |
 \le Ct \| \partial_t u_\epsilon \|_6 \| {\mathcal{L}}_\epsilon^{\frac 12} \psi^\epsilon \|_3
 \|{\mathcal{L}}_\epsilon^{\frac 12} t  \tilde P^\epsilon \|_2 \nonumber\\
 & \le  C t
 \| \partial_t u_\epsilon \|_{H^1} \| {\mathcal{L}}_\epsilon^{\frac 12} \psi^\epsilon \|_2^{\frac 12}
 \| {\mathcal{L}}_\epsilon^{\frac 12} \psi^\epsilon \|_{H^1}^{\frac 12}
 \|{\mathcal{L}}_\epsilon^{\frac 12} t  \tilde P^\epsilon \|_2 \nonumber\\
 & \le {\frac 1{32}} \|{\mathcal{L}}_\epsilon^{\frac 12} t  \tilde P^\epsilon \|^2_{H_\epsilon}
 + Ct^2  \|\mathcal{L}_\epsilon^{\frac 12} \partial_t u_\epsilon \|_{H_\epsilon} ^2
 \| {\mathcal{L}}_\epsilon^{\frac 12} \psi^\epsilon \|_{H_\epsilon}
 \| {\mathcal{L}}_\epsilon \psi^\epsilon \|_{H_\epsilon}  \nonumber\\
  & \le {\frac 1{32}} \|{\mathcal{L}}_\epsilon^{\frac 12} t  \tilde P^\epsilon \|^2_{H_\epsilon}
 + \sqrt{\epsilon} C e^{Ct}   \| t \mathcal{L}_\epsilon^{\frac 12} \partial_t u_\epsilon \|_{H_\epsilon} ^2.
   \end{align}
   By \eqref{pflemma435}, the second term on the
   right-hand side of \eqref{pflemma433} is less than
   \begin{equation}
   \label{pflemma437}
   \alpha_1 t  \left ( (u_\epsilon -L_1^\epsilon) {\mathcal{L}}^{\frac 12}_\epsilon \tilde{\psi}^\epsilon,
 {\mathcal{L}}^{\frac 12}_\epsilon t \tilde P^{\epsilon} \right )_{H_\epsilon}
 \le  {\frac 1{32}} \|{\mathcal{L}}_\epsilon^{\frac 12} t  \tilde P^\epsilon \|^2_{H_\epsilon}
 + C \left ( \| t \tilde P^\epsilon \|^2_{H_\epsilon}
  +\| t \tilde Q^\epsilon \|^2_{H_\epsilon} \right ) +   \epsilon^2 C t^2.
   \end{equation}
By Lemma \ref{lemma36}, the fourth term on the
   right-hand side of \eqref{pflemma433} is bounded by
   \begin{equation}
   \label{pflemma438}
   C \| t \mathcal{L}_\epsilon^{\frac 12} \phi_t \|_\infty \| P^\epsilon \|_2
   \|{\mathcal{L}}_\epsilon^{\frac 12} t  \tilde P^\epsilon \|_2
   \le {\frac 1{32}} \|{\mathcal{L}}_\epsilon^{\frac 12} t  \tilde P^\epsilon \|^2_{H_\epsilon}
+ C \| t  \phi_t \|_{H^2}^2 \| P^\epsilon \|^2_{H_\epsilon}
 \le {\frac 1{32}} \|{\mathcal{L}}_\epsilon^{\frac 12} t  \tilde P^\epsilon \|^2_{H_\epsilon}
+ \epsilon Ce^{Ct}.
 \end{equation}
 Other terms on the right-hand side of \eqref{pflemma433} can be estimated  in a similar way
as the proof of Lemma \ref{lemma42}. Therefore, by  \eqref{pflemma433},
 \eqref{pflemma436}-\eqref{pflemma438} and the estimates for other terms, we have
    \begin{align}\label{pflemma439}
    {\frac 1{2 D_1}} {\frac d{dt}} \| t \tilde{P}^\epsilon \|_{H_\epsilon}^2
    + {\frac 34} \| \mathcal{L}_\epsilon^{\frac 12} t \tilde P^\epsilon \|^2_{H_\epsilon}
     \le  &C \left (\| t \tilde{P}^\epsilon \|_{H_\epsilon}^2 +\| t \tilde{Q}^\epsilon \|_{H_\epsilon}^2
     \right ) + \epsilon C e^{Ct} + \epsilon^2  C  \left ( \| t u_t \|_{H^1}^2
     + \| t v_t \|_{H^1}^2 \right ) \nonumber\\
     &+ \sqrt{\epsilon} C e^{Ct} \| t \mathcal{L}_\epsilon^{\frac 12} \partial_t u_\epsilon \|^2_{H_\epsilon}+ {\frac 1{D_1}} (\tilde P^\epsilon, t \tilde P^\epsilon )_{H_\epsilon}.
    \end{align}
Next, we deal with the last term on the right-hand side  of the above inequality. Replacing $P$
in  \eqref{diffeq2}     by $t \partial_t P^\epsilon = t\tilde P^\epsilon $, we get
    \begin{align*}
    {\frac 1{D_1}} (\tilde P^\epsilon, t \tilde P^\epsilon )_{H_\epsilon}
    &+ {\frac 12} t {\frac d{dt}} \| \mathcal{L}_\epsilon^{\frac 12} P^\epsilon \|^2_{H_\epsilon}
     =
 \alpha_1\left((u_{\epsilon}-L_1^\epsilon) {\mathcal{L}}^{\frac 12}_{\epsilon}  \psi^{\epsilon},
 {\mathcal{L}}^{\frac 12}_{\epsilon}  t \tilde P^\epsilon  \right)_{H_{\epsilon}}\nonumber\\
&+\alpha_1\left(P^{\epsilon}
 {\mathcal{L}}^{\frac 12}_{\epsilon} (\phi-L^0_3),
  {\mathcal{L}}^{\frac 12}_{\epsilon}  t \tilde P^\epsilon
   \right)_{H_{\epsilon}}-G_1(u-L_1^0,\phi-L_3^0, t\tilde P^{\epsilon}).
\end{align*}
Using Lemma \ref{lemma42} and proceeding  as before, we obtain from the above that
    \begin{equation}
    \label{pflemma4310}
    {\frac 1{D_1}} (\tilde P^\epsilon, t \tilde P^\epsilon )_{H_\epsilon}
   + {\frac 12} t {\frac d{dt}} \| \mathcal{L}_\epsilon^{\frac 12} P^\epsilon \|^2_{H_\epsilon}
  \le    {\frac 14} \| {\mathcal{L}} _\epsilon^{\frac 12} t \tilde P ^\epsilon \|^2_{H_\epsilon}
     + \epsilon C e^{Ct} + \epsilon^2 C.
     \end{equation}
     Then it follows from \eqref{pflemma439}-\eqref{pflemma4310} that
      \begin{align*}
    {\frac 1{2 D_1}} {\frac d{dt}} \| t \tilde{P}^\epsilon \|_{H_\epsilon}^2
    + & {\frac 12} \| \mathcal{L}_\epsilon^{\frac 12} t \tilde P^\epsilon \|^2_{H_\epsilon}
    + {\frac 12} t {\frac d{dt}} \| \mathcal{L}_\epsilon^{\frac 12} P^\epsilon \|^2_{H_\epsilon} \nonumber\\
      & \le   C \left (
     \| t \tilde{P}^\epsilon \|_{H_\epsilon}^2 +\| t \tilde{Q}^\epsilon \|_{H_\epsilon}^2
     \right ) + \epsilon C e^{Ct}  \nonumber\\
     & + \epsilon^2  C  \left ( \| t u_t \|_{H^1}^2
     + \| t v_t \|_{H^1}^2 \right )
     + \sqrt{\epsilon} C e^{Ct} \| t \mathcal{L}_\epsilon^{\frac 12} \partial_t u_\epsilon \|^2_{H_\epsilon},
      \end{align*}
    which implies that
     \begin{align}
     \label{pflemma4311}
    {\frac 1{2}} {\frac d{dt}} \left (
    {\frac 1{D_1}}  \| t \tilde{P}^\epsilon \|_{H_\epsilon}^2
    +  t  \| \mathcal{L}_\epsilon^{\frac 12} P^\epsilon \|^2_{H_\epsilon}
    \right )
      & \le   {\frac 12} \| \mathcal{L}_\epsilon^{\frac 12} P^\epsilon \|^2_{H_\epsilon}
      +     C \left (
     \| t \tilde{P}^\epsilon \|_{H_\epsilon}^2 +\| t \tilde{Q}^\epsilon \|_{H_\epsilon}^2
     \right ) + \epsilon C e^{Ct}  \nonumber\\
     & + \epsilon^2  C  \left ( \| t u_t \|_{H^1}^2
     + \| t v_t \|_{H^1}^2 \right )
     + \sqrt{\epsilon} C e^{Ct} \| t \mathcal{L}_\epsilon^{\frac 12} \partial_t u_\epsilon \|^2_{H_\epsilon}.
      \end{align}
  Similarly, by equation \eqref{pflemma434}, we can show that
        \begin{align}
     \label{pflemma4312}
    {\frac 1{2}} {\frac d{dt}} \left (
    {\frac 1{D_2}}  \| t \tilde{Q}^\epsilon \|_{H_\epsilon}^2
    +  t  \| \mathcal{L}_\epsilon^{\frac 12} Q^\epsilon \|^2_{H_\epsilon}
    \right )
      & \le   {\frac 12} \| \mathcal{L}_\epsilon^{\frac 12} Q^\epsilon \|^2_{H_\epsilon}
      +     C \left (
     \| t \tilde{P}^\epsilon \|_{H_\epsilon}^2 +\| t \tilde{Q}^\epsilon \|_{H_\epsilon}^2
     \right ) + \epsilon C e^{Ct}  \nonumber\\
     & + \epsilon^2  C  \left ( \| t u_t \|_{H^1}^2
     + \| t v_t \|_{H^1}^2 \right )
     + \sqrt{\epsilon} C e^{Ct} \| t \mathcal{L}_\epsilon^{\frac 12} \partial_t v_\epsilon \|^2_{H_\epsilon}.
      \end{align}
      By \eqref{pflemma4311}-\eqref{pflemma4312} we find that
       \begin{align*}
       &  {\frac d{dt}} \left ({\frac 1{D_1}}  \| t \tilde{P}^\epsilon \|_{H_\epsilon}^2
    +  t  \| \mathcal{L}_\epsilon^{\frac 12} P^\epsilon \|^2_{H_\epsilon} +
      {\frac 1{D_2}}  \| t \tilde{Q}^\epsilon \|_{H_\epsilon}^2
    +  t  \| \mathcal{L}_\epsilon^{\frac 12} Q^\epsilon \|^2_{H_\epsilon} \right ) \\
      & \qquad  \le   C  \left (
       {\frac 1{D_1}}  \| t \tilde{P}^\epsilon \|_{H_\epsilon}^2
    +  t  \| \mathcal{L}_\epsilon^{\frac 12} P^\epsilon \|^2_{H_\epsilon} +
      {\frac 1{D_2}}  \| t \tilde{Q}^\epsilon \|_{H_\epsilon}^2
    +  t  \| \mathcal{L}_\epsilon^{\frac 12} Q^\epsilon \|^2_{H_\epsilon} \right )
       +  \| \mathcal{L}_\epsilon^{\frac 12} P^\epsilon \|^2_{H_\epsilon}
      +\| \mathcal{L}_\epsilon^{\frac 12} Q^\epsilon \|^2_{H_\epsilon}\\
       & \qquad + \epsilon C e^{Ct} + \epsilon^2  C  \left ( \| t u_t \|_{H^1}^2
     + \| t v_t \|_{H^1}^2 \right ) + \sqrt{\epsilon} C e^{Ct}
      \left ( \| t \mathcal{L}_\epsilon^{\frac 12} \partial_t u_\epsilon \|^2_{H_\epsilon}
     +\| t \mathcal{L}_\epsilon^{\frac 12} \partial_t v_\epsilon \|^2_{H_\epsilon} \right ),
      \end{align*}
      which, along with    Gronwall's lemma  and Lemmas \ref{lemma35}, \ref{lemma36}
      and \ref{lemma42}, implies Lemma \ref{lemma43}.       \end{proof}

      Let $(c_1^\epsilon, c_2^\epsilon, \Phi^\epsilon)$
      be the solutions of problem \eqref{newPNP}-\eqref{newBV}
      with initial datum       $(c_{1,0}, c_{2,0})$, and
      $(c_1, c_2, \Phi)$ be the solutions of
       problem \eqref{genlimPNP}-\eqref{genlimBV} with
       initial datum $(M(c_{1,0}),  M(c_{2,0}))$.
      Then as an immediate consequence of Lemma~\ref{lemma43}, we find the
      following estimates which are essential to prove the upper semi-continuity
      of the global attractors.

\begin{lem}
\label{lemma44}
There exists    $\epsilon_1 >0 $ such that, for
any $R>0$, there exists a constant $K$ depending on $R$ such that,
for any $0<\epsilon \le \epsilon_1$ and  $(c_{1,0}, c_{2,0})\in
{\tilde{\Sigma}} $ with $\|  (c_{1,0}, c_{2,0})   \|_{X_\epsilon \times
X_\epsilon} \le R$, the following holds:
\[   \left (  \|c_1^\epsilon (t)  -c_1 (t)  \|^2_{X_\epsilon}
    +    \|   c_2^\epsilon (t)   -c_2 (t) \|^2_{X_\epsilon}  \right )
    \le \sqrt{\epsilon} Ke^{Kt},  \quad t \ge 1.\]
  \end{lem}

We  are now in a position to prove the upper semi-continuity of global attractors.

\noindent
{\bf Proof of Theorem~\ref{uppersc}}. \
Let $T^\epsilon(t)_{t \ge 0}$ and $T^0(t)_{t\ge 0}$ be
the solution operators of problem  \eqref{newPNP}-\eqref{newBV}
     and
       problem \eqref{genlimPNP}-\eqref{genlimBV}, respectively.
       Then it follows from Proposition~\ref{proposition33} that  there is a constant $R>0$ (independent of
$\epsilon$) such that
\[  \| (c_1, c_2)  \|_{X_{\epsilon} \times X_\epsilon} \le R,\; \mbox{ for all }\; (c_1, c_2) \in
{\mathcal A}_{\epsilon}.\]
For the given $\eta>0$,
since ${\mathcal A}_0$ is the  global attractor of $T^0(t)$, there exists
$\tau_0=\tau_0(\eta,R)\ge 1 $ such that, for any $t\ge \tau_0$,
\[\inf_{z_0\in {\mathcal A}_0}\|T^0(t)(Mz)-z_0\|_{X_\epsilon \times X_\epsilon}\le \frac{\eta}{2},\]
for any $z=(c_1, c_2)\in {\mathcal A}_{\epsilon}$.
On the other hand, by Lemma~\ref{lemma44} we find that
   \[\|T^{\epsilon}( \tau_0)z -T^0( \tau_0)(Mz )\|_{X_{\epsilon} \times X_\epsilon }\le
   \epsilon^{\frac 14} K(R)e^{K(R)\tau_0},\]
 for some constant $K(R)$. Therefore, we obtain that, for any $z=(c_1, c_2)\in {\mathcal A}_{\epsilon}$:
  \[\inf_{z_0\in {\mathcal
 A}_0}\|T^{\epsilon}(\tau_0)z -z_0\|_{X_{\epsilon} \times X_\epsilon }\le
 \frac{\eta}{2}+
 \epsilon^{\frac 14} K(R)e^{K(R)\tau_0},\]
 which implies that, for $\epsilon>0$ small enough:
 \[ {\rm dist}_{X_{\epsilon} \times X_{\epsilon}}  \left (T^{\epsilon}(\tau_0) {\mathcal A}_{\epsilon},  {\mathcal A}_0 \right )\le \eta.\]
The proof is completed since $T^{\epsilon}(\tau_0) {\mathcal A}_{\epsilon} =  {\mathcal A}_{\epsilon}. $

\section{Steady-states for the one-dimensional limiting PNP system}
\setcounter{equation}{0}

Many mathematical works have been done on the existence, uniqueness and
qualitative properties of boundary value problems even
for high dimensional  systems and  algorithms have been developed
toward numerical approximations (see, e.g.~\cite{Hom, Jer, PJ, JeK}).
Under the assumption that $\mu\ll 1$, the problem can be viewed as
a singularly perturbed system. Typical solutions of singularly perturbed
systems exhibit different time scales; for example, boundary and internal
layers (inner solutions) evolve at fast pace and regular layers
(outer solutions) vary slowly. For the boundary value problem~\eqref{SSPNP}
and~\eqref{SSBV},   there are two boundary layers one at each end. Physically,
near boundaries $x=0$ and $x=1$, the potential function $\phi(x)$ and the
concentration functions $c_1(x)$ and $c_2(x)$ exhibit a large gradient or a
sharp change. In~\cite{BCEJ}, for $\alpha_1=\alpha_2=1$,
the boundary value problem for the direct (with $h(x)=1$ in \eqref{SSPNP}) one-dimensional PNP system  was studied   using the method
of matched asymptotic expansions as well as numerical
simulations, which provide a good quantitative and qualitative
 understanding of the problem.  In~\cite{Liu1},   geometric singular perturbation theory
(see, e.g.~\cite{Fen, Jon, JK, Liu}) was applied to the study of this singular boundary value problem.
The treatment for the limiting one-dimensional PNP system \eqref{SSPNP}  carrying  the geometric information of the three-dimensional channel follows that in~\cite{Liu1}.

It is convenient to study an equivalent connecting problem to the boundary value problem~\eqref{SSPNP} and~\eqref{SSBV}.
 Let $B_L$ and $B_R$ be the subsets of $\bbR^7$ defined, respectively, by
\begin{align}\label{left-right}\begin{split}
B_L=&\{\phi=\phi_0,\;v=-h(0)(\alpha_1 l_1-\alpha_2 l_2),\;w=\alpha_1^2 l_1+\alpha_2^2 l_2,\;
\tau=0\},\\
B_R=&\{ \phi=0,\;v=-h(1)(\alpha_1 r_1-\alpha_2 r_2),\;w=\alpha_1^2 r_1+\alpha_2^2r_2,\;
\tau=1\}.
\end{split}
\end{align}
The boundary value problem is then equivalent to the following
{\em connecting problem}: finding a solution of~\eqref{slow} from
$B_L$ to $B_R$.

For $\mu>0$, let $M_L^{\mu}$ be the  union  of all forward orbits
of~\eqref{slow} starting from $B_L$ and let $M_R^{\mu}$ be the union  of
all backward  orbits starting from $B_R$. To obtain the existence and (local)
uniqueness of a solution for the connecting problem,   it thus suffices to show
$M_L^{\mu}$ and $M_R^{\mu}$ intersect  transversally.   The
intersection is exactly the orbit of  a solution of the boundary value problem,
and the transversality implies  the local uniqueness.  The strategy is to
obtain a singular orbit and track the evolution of $M_L^{\mu}$ and
$M_R^{\mu}$ along the singular orbit. As discussed in the introduction, a
singular orbit will be  a union of orbits of subsystems of~\eqref{slow} with
different time scales.

The   boundary layers will be two orbits of~\eqref{limfast}:
one from $B_L$ to ${\mathcal Z}_0$ in forward time along the stable manifold
of ${\mathcal Z}_0$ and  the other from  $B_R$ to ${\mathcal Z}_0$ in backward
time along the unstable manifold of ${\mathcal Z}_0$. The two boundary layers
will be connected by a regular layer  on ${\mathcal Z}_0$, which is an orbit of
a limiting system of~\eqref{slow}.  The next two subsections are devoted to the
study of boundary layers and regular layers.

\subsection{Fast dynamics and boundary layers}\label{inner}
We start with the study of boundary layers governed by system~\eqref{limfast}.
This system has many invariant structures  that are useful for characterizing
the global dynamics.

The slow manifold ${\mathcal Z}_0=\{u=v=0\}$ consisting
of entirely equilibria of system~\eqref{limfast} is a $5$-dimensional
manifold of the phase space $\bbR^7$. For each  equilibrium
$z=(\phi, 0, 0, w, J_1,J_2, \tau)\in {\mathcal Z}_0$, the linearization of
system~\eqref{limfast} has five zero eigenvalues corresponding to the
dimension of ${\mathcal Z}_0$, and  two  eigenvalues in directions normal
to ${\mathcal Z}_0$. The latter two eigenvalues and their associated
eigenvectors  are given  by
\begin{equation}\label{e-value}
\lambda_{\pm}=\pm\sqrt{w} \;\mbox{ and }\;
n_{\pm}=\left((\pm\sqrt{w})^{-1},1,\pm\sqrt{w},
\pm(\alpha_2-\alpha_1)\sqrt{w}, 0, 0, 0\right)^{\tau}.
\end{equation}
Thus, every equilibrium has a one-dimensional  stable manifold and a
one-dimensional unstable manifold. The global configurations of the stable and
unstable manifolds will be needed for the boundary layer behavior.
For any constants $J_1^*$, $J_2^*$ and
$\tau^*$, the set ${\mathcal N}=\{J_1=J_1^*, J_2=J_2^*, \tau=\tau^*\}$ is a
$4$-dimensional invariant subspace of the phase space $\bbR^7$.

Surprisingly,   system~\eqref{limfast} possesses   a complete set of integrals
with which the dynamics  can be fully analyzed; in particular,
the stable and unstable manifolds can be  characterized and
the behavior of boundary layers can be described in detail.

\begin{prop}\label{s-u-inter} (i) System~\eqref{limfast} has a complete set
of six integrals given by
\begin{align*}\label{int}\begin{split}
 H_1&= w-\frac{\alpha_2-\alpha_1}{h(\tau)}v-\frac{\alpha_1\alpha_2}{2h^2(\tau)}u^2,\;
H_2=\phi-\frac{\ln |\alpha_1 v/{h(\tau)}+w|}{\alpha_2}, \\
  H_3=&\phi+\frac{\ln |\alpha_2 v/{ h(\tau)}-w|}{\alpha_1},\;
 H_4=J_1,\; H_5=J_2 \;\mbox{ and }\;  H_6=\tau,
\end{split}
\end{align*}
where the argument of $H_i$'s is $(\phi,u,v,w,J_1,J_2,\tau)$.

(ii) The stable and unstable manifolds $W^s({\mathcal Z}_0)$ and
$W^u({\mathcal Z}_0)$ of ${\mathcal Z}_0$ are characterized as follows:
\[W^s({\mathcal Z}_0)=\cup\{W^s(z^*):\,z^*\in {\mathcal Z}_0\}\;\mbox{ and }\;
 W^u({\mathcal Z}_0)=\cup\{W^u(z^*):\,z^*\in {\mathcal Z}_0\}\]
 and, for $z^*=(\phi^*,0,0,w^*,J_1^*,J_2^*,\tau^*)\in {\mathcal Z}_0$,
 a point $z=(\phi,u,v,w,J_1,J_2,\tau)\in W^s(z^*)\cup W^u(z^*)$ if and only if
\begin{align*}
H_1(z)= &w^*,\;H_2(z)=\phi^*-\frac{\ln w^*}{\alpha_2},
 \;H_3(z)=\phi^*+\frac{\ln w^*}{\alpha_1}, \\
 J_1= & J_1^*, \; J_2=  J_2^*, \;\tau=\tau^*.
 \end{align*}

(iii) The  stable  manifold $W^s({\mathcal Z}_0)$ intersects $B_L$
transversally at points with
\begin{equation}\label{s-U}
u=-[\mbox{sgn }(\alpha_2 l_2-\alpha_1 l_1)]\sqrt{2}h(0)
\sqrt{ l_1+l_2-\frac{(\alpha_1+\alpha_2)
(\alpha_1 l_1)^{\frac{\alpha_2}{\alpha_1+\alpha_2}}
(\alpha_2 l_2)^{\frac{\alpha_1}{\alpha_1+\alpha_2}} }
{\alpha_1\alpha_2}}
\end{equation}
and arbitrary $J_1$ and $J_2$, where $\mbox{sgn}$ is the sign function. The
unstable  manifold $W^u({\mathcal Z}_0)$ intersects  $B_R$ transversally
at points with
\begin{equation}\label{u-U}
u=[\mbox{sgn }(\alpha_2 r_2-\alpha_1 r_1)]\sqrt{2}h(1)
\sqrt{r_1+r_2-\frac{(\alpha_1+\alpha_2)
(\alpha_1 r_1)^{\frac{\alpha_2}{\alpha_1+\alpha_2}}
(\alpha_2 r_2)^{\frac{\alpha_1}{\alpha_1+\alpha_2}} }
{\alpha_1\alpha_2}}\end{equation}
and arbitrary $J_1$ and $J_2$. Let $N_L=B_L\cap W^s({\mathcal Z}_0)$ and
$N_R=B_R\cap W^u({\mathcal Z}_0)$.   Then,
\[\omega(N_L)=\left\{\left(\phi_0+\frac{1}{\alpha_1+\alpha_2}
   \ln\frac{\alpha_1 l_1}{\alpha_2 l_2},
 0, 0,(\alpha_1+\alpha_2)(\alpha_1 l_1)^{\frac{\alpha_2}{\alpha_1+\alpha_2}}
(\alpha_2 l_2)^{\frac{\alpha_1}{\alpha_1+\alpha_2}},J_1,J_2, 0\right)\right\},\]
\[\alpha(N_R)=\left\{\left(\frac{1}{\alpha_1+\alpha_2}
  \ln\frac{\alpha_1 r_1}{\alpha_2 r_2},
 0, 0,(\alpha_1+\alpha_2)(\alpha_1 r_1)^{\frac{\alpha_2}{\alpha_1+\alpha_2}}
(\alpha_2 r_2)^{\frac{\alpha_1}{\alpha_1+\alpha_2}},J_1,J_2, 1\right)\right\}\]
for all $J_1$ and $J_2$.
 \end{prop}
\begin{proof} The statement (i) can be verified directly.
The statement~(ii) is a simple consequence of~(i) together with the fact that
$\phi(\xi)\to \phi^*$, $w(\xi)\to w^*$, $u(\xi)\to 0$ and $v(\xi)\to 0$
as $\xi\to \infty$ for the stable manifold and
as $\xi\to -\infty$ for the unstable manifold.

For the statement (iii), we present only the proof regarding the intersection
of  $W^s({\mathcal Z}_0)$ and $B_L$.   Suppose
\begin{align*}
z^0=(\phi^0,u^0,v^0,w^0,J^0_1,J^0_2,0)
=(\phi_0, u^0,h(0)(\alpha_2 l_2-\alpha_1 l_1), \alpha_1^2 l_1+\alpha_2^2 l_2, J^0_1,J^0_2,0)
\end{align*}
is a point in $B_L\cap W^s({\mathcal Z}_0)$. Then, using the integrals $H_1$,
$H_2$ and $H_3$, the solution
$z(\xi)=(\phi(\xi),u(\xi),v(\xi),w(\xi),J^0_1,J^0_2,0)$
of system~\eqref{limfast} with  initial condition $z(0)=z^0$ satisfies
\begin{align*}
H_1(z(\xi))=& w(\xi)-\frac{\alpha_2-\alpha_1}{h(0)}v(\xi)
   -\frac{\alpha_1\alpha_2}{2h^2(0)}u^2(\xi)=A,\\
H_2(z(\xi))=&\phi(\xi)-\frac{1}{\alpha_2}\ln |\alpha_1 v(\xi)/{h(0)}+w(\xi)|=B,\\
H_3(z(\xi))=&\phi(\xi)+\frac{1}{\alpha_1}\ln |\alpha_2 v(\xi)/{h(0)}-w(\xi)| =C
\end{align*}
for some constants $A$, $B$ and $C$, and for all $\xi$. From the initial
condition, we get
\[B=\phi_0-\frac{\ln (\alpha_1+\alpha_2)}{\alpha_2}-\frac{\ln (\alpha_2 l_2)}{\alpha_2}
\;\mbox{ and }\;
C=\phi_0+\frac{\ln (\alpha_1+\alpha_2)}{\alpha_1}+\frac{\ln (\alpha_1 l_1)}{\alpha_1}.\]
Since  $u(\xi)\to 0$ and $v(\xi)\to 0$ as $\xi\to +\infty$, we have that
$ w(+\infty)=A$  and
\[C-B=\frac{\alpha_1+\alpha_2}{\alpha_1\alpha_2}\ln w(+\infty)
=\frac{\alpha_1+\alpha_2}{\alpha_1\alpha_2}\ln A.\]
Hence,
\[w(+\infty)=A=(\alpha_1+\alpha_2)(\alpha_1 l_1)^{\frac{\alpha_2}{\alpha_1+\alpha_2}}
(\alpha_2 l_2)^{\frac{\alpha_1}{\alpha_1+\alpha_2}}.\]
Therefore,
\[u^0=-[\mbox{sgn }(v^0)]\sqrt{2}h(0)
  \sqrt{ l_1+l_2-\frac{A}{\alpha_1\alpha_2}} \;\mbox{ and }
\;\phi(+\infty)=\phi_0+\frac{1}{\alpha_1+\alpha_2}\ln\frac{\alpha_1 l_1}{\alpha_2 l_2}.\]
The choice of the sign for $u^0$ comes from the consideration that
the stable eigenvector $n_-$ in~\eqref{e-value} has $u$ and $v$ components
with opposite signs.
Thus, $B_L$ and $W^s({\mathcal Z}_0)$ intersect at the points with $u=u^0$
given above, and all $J_1$ and $J_2$. If $N_L=B_L\cap W^s({\mathcal Z}_0)$,
then $\omega(N_L)=\{(\phi(+\infty), 0,0,w(+\infty), J_1,J_2,0)\}$. The
above formulas for $\phi(+\infty)$ and $w(+\infty)=A$ gives the desired
characterization of  $\omega(N_L)$. Lastly, since the stable manifold is
completely characterized, one can compute its tangent space at each
intersection point to verify the transversality of the intersection.
It is slightly complicated but straightforward. We will omit the detail here.
\end{proof}

Part (iii) of this result implies that the boundary layer on the left end will be
an orbit of~\eqref{limfast} from $(\phi_0,u_L, \alpha_2 l_2-\alpha_1 l_1,
\alpha_1^2l_1+ \alpha_2^2l_2, J_1, J_2, 0)\in B_L$ to the point
  \[z_L=\left(\phi_0+\frac{1}{\alpha_1+ \alpha_2}
   \ln\frac{\alpha_1 l_1}{  \alpha_2 l_2},
 0, 0,(\alpha_1+  \alpha_2)(\alpha_1 l_1)^{\frac{ \alpha_2}{\alpha_1+ \alpha_2}}
(  \alpha_2 l_2)^{\frac{\alpha_1}{\alpha_1+  \alpha_2}},J_1,J_2, 0\right)\in {\mathcal Z}_0,\]
where $U_L$ is given by the display~\eqref{s-U} and $I_1$ and $I_2$ are
arbitrary at this moment; and that on the right end will be
a backward orbit of~\eqref{limfast} from the point
$(0,u_R,   \alpha_2 r_2-\alpha_1 r_1,
\alpha_1^2r_1+  \alpha_2^2r_2, J_1, J_2, 1)\in B_R$ to the point
\[z_R= \left(\frac{1}{\alpha_1+  \alpha_2}
  \ln\frac{\alpha_1 r_1}{  \alpha_2r_2},
 0, 0,(\alpha_1+  \alpha_2)(\alpha_1 r_1)^{\frac{  \alpha_2}{\alpha_1+  \alpha_2}}
(  \alpha_2 r_2)^{\frac{\alpha_1}{\alpha_1+ \alpha_2}},J_1,J_2, 1\right)\in {\mathcal Z}_0,\]
where $u_R$ is given by the display~\eqref{u-U} and $J_1$ and $J_2$ are
arbitrary at this moment.  It turns out that there is a unique
pair of numbers $J_1$ and $J_2$ so that the corresponding points $z_L$ and $z_R$
can be connected by a regular layer solution on ${\mathcal Z}_0$. The regular
orbit together with the two  boundary layer orbits provide the singular orbit.

\subsection{Slow dynamics and regular layers}\label{outer}
 We  now examine the slow flow in the vicinity of the
slow manifold ${\mathcal Z}_0=\{u=v=0\}$ for regular layers.
Note that system \eqref{limslow} resulting from \eqref{slow} by setting   $\mu=0$  reduces to
$u=v=0$ and
\[\dot J_1=0,\; \dot J_2=0,\; \dot \tau=1.\]
The information on $\phi$ and $w$ is lost. This indicates that
the slow flow in the vicinity of  ${\mathcal Z}_0$  is itself a
singular perturbation problem.
 To see this,  we zoom into an $O(\mu)$-neighborhood of ${\mathcal Z}_0$
by blowing up the $u$ and $v$ coordinates; that is, we make a scaling
$u=\mu p$ and $v=\mu q$. System~\eqref{slow}  becomes
\begin{align}\label{newslow}\begin{split}
  \dot \phi=& \frac{1}{h(\tau)}p,\quad
  \mu \dot p= q, \quad
  \mu \dot q=pw+\mu\frac{h_{\tau}(\tau)}{h(\tau)}q +(\alpha_1 J_1-  \alpha_2 J_2),\\
\dot w=& \mu \frac{\alpha_1 \alpha_2 }{h^2(\tau)} pq
   +\frac{\alpha_2-\alpha_1}{h(\tau)}pw-\frac{\alpha_1^2J_1+  \alpha_2^2J_2}{h(\tau)},\\
 \dot J_1=& 0,\quad \dot J_2=0,\quad \dot \tau=1,
 \end{split}
 \end{align}
which is indeed a singular perturbation problem.  When $\mu=0$, the system
 reduces to
\begin{align}\label{limnewslow}\begin{split}
  \dot \phi=& \frac{1}{h(\tau)}p,\quad
  0=  q, \quad
   0=pw+(\alpha_1 J_1-  \alpha_2 J_2),\\
   \dot w=& \frac{\alpha_2-\alpha_1}{h(\tau)}pw
   -\frac{\alpha_1^2J_1+\alpha_2^2 J_2}{h(\tau)},\\
   \dot J_1=& 0,\quad \dot J_2=0,\quad \dot \tau=1.
 \end{split}
 \end{align}
Dynamics of $\phi$ and $w$ survives in this limiting process.
 For this system, the slow manifold is
\[{\mathcal S}_0=\left\{p=\frac{\alpha_2 J_2-\alpha_1 J_1}{w},\; q=0\right\}.\]
The corresponding fast system   is
 \begin{align}\label{newfast}\begin{split}
  \phi'=&  \mu \frac{p}{h(\tau)},\quad
   p'=   q,\quad
   q'=pw+(\alpha_1 J_1- \alpha_2 J_2)+\mu\frac{h_{\tau}(\tau)}{h(\tau)}q,\\
   w'=&\mu^2\frac{\alpha_1 \alpha_2 }{h^2(\tau)}
  pq+\mu\frac{ \alpha_2-\alpha_1}{h(\tau)}pw
-\mu\frac{\alpha_1^2J_1+  \alpha_2^2J_2}{h(\tau)},\\
   J_1'=&0,\quad J_2'=0,\quad \tau'=0.
 \end{split}
 \end{align}
The limiting system of~\eqref{newfast}  when $\mu= 0$ is
\begin{align}\label{limnewfast}\begin{split}
  \phi'=& 0,\quad
   p'=   q,\quad
   q'=pw+(\alpha_1 J_1-\alpha_2 J_2),\\
   w'= &0,\quad
   J_1'=0,\quad J_2'=0,\quad \tau'=0.
 \end{split}
 \end{align}
 The slow manifold ${\mathcal S}_0$ is the set of equilibria
 of~\eqref{limnewfast}. The eigenvalues normal to ${\mathcal S}_0$
are  $\lambda_{\pm}(p)=\pm\sqrt{w}$.  In particular, the slow manifold
${\mathcal S}_0$ is normally hyperbolic, and hence, it persists for
system~\eqref{newfast} for $\mu>0$ small (see~\cite{Fen}).

The limiting slow dynamic on ${\mathcal S}_0$ is governed by
system~\eqref{limnewslow}, which reads
  \[\dot \phi= \frac{\alpha_2J_2-\alpha_1 J_1}{h(\tau)w},
  \quad \dot w=-\frac{\alpha_1 \alpha_2 (J_1+J_2)}{h(\tau)},\quad \dot J_i=0,
  \quad \dot \tau=1.\]
 The general solution is characterized as: $J_1$ and $J_2$ are
 arbitrary constants, and
\begin{align}\label{genslow}\begin{split}
\tau(x)=&\tau_0+x,\quad
w(x)=w_0-\alpha_1\alpha_2(J_1+J_2)\int_0^x\frac{1}{h(\tau_0+s)}\,ds,\\
\phi(x)=&\nu_0+ (\alpha_2 J_2-\alpha_1 J_1)\int_0^x\frac{1}{h(\tau_0+s)w(s)}\,ds\\
=&\nu_0-\frac{\alpha_2 J_2-\alpha_1 J_1}{\alpha_1 \alpha_2(J_1+J_2)}
\ln\left|1-\frac{\alpha_1 \alpha_2 (J_1+J_2)}{w_0}\int_0^x\frac{1}{h(\tau_0+s)}\,ds\right|.
\end{split}
\end{align}
where $\tau_0=\tau(0)$, $\phi(0)=\nu_0$ and $w(0)=w_0$.
Note that, if $J_1+J_2=0$, then $w(x)=w_0$ and
\[\phi(x)=\nu_0+\frac{  \alpha_2 J_2-\alpha_1J_1}{w_0}\int_0^x\frac{1}{h(\tau_0+s)}\,ds.\]
 The latter is the limit
of $\phi(x)$ in~\eqref{genslow} as $J_1+J_2\to 0$. We thus use the unified
formula~\eqref{genslow}  even if $J_1+J_2=0$.

To identify the slow portion of the singular orbit on ${\mathcal S}_0$, we need
to examine the $\omega$-limit (resp. the $\alpha$-limit) set of
$M^{\mu}_L\cap W^s({\mathcal S}_0)$
(resp. $M^{\mu}_R\cap W^u({\mathcal S}_0)$) as $\mu\to 0$.
To do this, we fix an $O(1)$-neighborhood of ${\mathcal S}_0$.
In terms of $U$ and $V$, this neighborhood is of order $O(\mu)$. For
$\mu>0$ small, the time taken in terms of $\xi$ for $M^{\mu}_L$ and
$M^{\mu}_R$ to evolve to any $O(\mu)$-neighborhood of $\{u=v=0\}$ is
of order $O(\mu|\ln\mu|)$. Thus, the $\lambda$-Lemma~(\cite{Den})
implies that
$M^{\mu}_L$ (resp. $M^{\mu}_R$) is $C^1$ $O(\mu)$-close to
$M^{0}_L$ (resp. $M^{0}_R$) in any $O(\mu)$-neighborhood of $\{u=v=0\}$.
Therefore, in an $O(1)$-neighborhood of ${\mathcal S}_0$ in terms of $p$ and
$q$, $M^{\mu}_L$ (resp. $M^{\mu}_R$) intersects $W^s({\mathcal S}_0)$
(resp. $W^u({\mathcal S}_0)$) transversally. And, by abusing the notations, if
 $N_L=M^{0}_L\cap W^s({\mathcal S}_0)$ and $N_R=M^{0}_R\cap W^u({\mathcal
 S}_0)$, then $\omega(N_L)$ and $\alpha(N_R)$ have the same descriptions
 as those in Proposition~\ref{s-u-inter} with $u=v=0$ replacing by
 $p=(\alpha_2J_2-\alpha_1 J_1)/w$ and $q=0$.

The slow orbit should be  one given by~\eqref{genslow}  that connects
$\omega(N_L)$ and $\alpha(N_R)$. Let  $\bar{M}_L$ (rep. $\bar{M}_R$)
be the forward (resp. backward) image
of $\omega(N_L)$ (resp. $\alpha(N_R)$) under the slow flow~\eqref{limnewslow}.
\begin{prop}\label{s-orbit}
 $\bar{M}_L$ and $\bar{M}_R$ intersect transversally along the
unique orbit given by~\eqref{genslow} from $x=0$ to $x=1$ with
 \[
\tau_0=0,\; w_0=(\alpha_1+  \alpha_2)(\alpha_1 l_1)^{\frac{\alpha_2}{\alpha_1+ \alpha_2}}
(\alpha_2 l_2)^{\frac{\alpha_1}{\alpha_1+  \alpha_2}},\;
\nu_0=\phi_0+\frac{1}{\alpha_1+ \alpha_2}\ln\frac{\alpha_1 l_1}{\alpha_2 l_2},\]
and $J_1$ and $J_2$ are as given in Theorem~\ref{main}.
\end{prop}
\begin{proof} We show first that  $\bar{M}_L$ and $\bar{M}_R$ intersect along
the orbit with the above characterization. In view of~\eqref{genslow} and the
descriptions for $\omega(N_L)$ and $\alpha(N_R)$ in Proposition~\ref{s-u-inter},
the intersection is uniquely determined by
\[\tau_0=0,\;
w(0)=(\alpha_1+  \alpha_2)(\alpha_1 l_1)^{\frac{  \alpha_2}{\alpha_1+ \alpha_2}}
(  \alpha_2 l_2)^{\frac{\alpha_1}{\alpha_1+ \alpha_2}},\]
\[w(1)=(\alpha_1+  \alpha_2)(\alpha_1 r_1)^{\frac{ \alpha_2}{\alpha_1+  \alpha_2}}
(  \alpha_2 r_2)^{\frac{\alpha_1}{\alpha_1+ \alpha_2}},\]
\[\nu_0=\phi(0)=\phi_0+\frac{1}{\alpha_1+ \alpha_2}\ln\frac{\alpha_1 l_1}{ \alpha_2 l_2},\;
\phi(1)=\frac{1}{\alpha_1+  \alpha_2}\ln\frac{\alpha_1 r_1}{\alpha_2 r_2}.\]
Substituting into~\eqref{genslow} gives
\begin{align*}
J_1+J_2=&\frac{\alpha_1+  \alpha_2}{\alpha_1 \alpha_2  \int_0^1h^{-1}(x)\,dx}
\left((\alpha_1 l_1)^{\frac{ \alpha_2}{\alpha_1+  \alpha_2}}
 (  \alpha_2 l_2)^{\frac{\alpha_1}{\alpha_1+  \alpha_2}}
 - (\alpha_1 r_1)^{\frac{ \alpha_2}{\alpha_1+  \alpha_2}}
 (  \alpha_2 r_2)^{\frac{\alpha_1}{\alpha_1+  \alpha_2}}\right),\\
  \alpha_2 J_2-\alpha_1 J_1=&\frac{(\alpha_1+  \alpha_2)
 \left((\alpha_1 l_1)^{\frac{  \alpha_2}{\alpha_1+  \alpha_2}}
 (  \alpha_2 l_2)^{\frac{\alpha_1}{\alpha_1+  \alpha_2}}
 - (\alpha_1 r_1)^{\frac{ \alpha_2}{\alpha_1+  \alpha_2}}
 (  \alpha_2 r_2)^{\frac{\alpha_1}{\alpha_1+  \alpha_2}}\right)}
{\left(\frac{  \alpha_2}{\alpha_1+  \alpha_2}\ln\frac{r_1}{l_1}+
\frac{\alpha_1}{\alpha_1+ \alpha_2}\ln\frac{r_2}{l_2}\right)\int_0^1h^{-1}(x)\,dx}\times\\
&\times\left(\phi_0+\frac{1}{\alpha_1+ \alpha_2}\ln\frac{l_1r_2}{l_2r_1}\right),
\end{align*}
which in turn yields the expressions for $J_1$ and $J_2$.
To see the transversality of the intersection, it suffices to show that
$\omega(N_L)\cdot 1 $ (the image of $\omega(N_L)$ under the time one map
of the flow of system~\eqref{limnewslow}) is transversal to $\alpha(N_R)$
on ${\mathcal S}_0\cap \{\tau=1\}$. If we use $(\phi, w, J_1,J_2)$ as
a coordinate system  on ${\mathcal S}_0\cap \{\tau=1\}$, then the set
 $\omega(N_L)\cdot 1 $ is  given  by $\{(\phi(J_1,J_2),w(J_1,J_2), J_1,J_2)\}$
with
\begin{align*}
\phi(J_1,J_2)=& \phi_0+\frac{1}{\alpha_1+ \alpha_2}\ln\frac{\alpha_1 l_1}{  \alpha_2 l_2}
-\frac{  \alpha_2 J_2-\alpha_1 J_1}{\alpha_1 \alpha_2 (J_1+J_2)}
   \ln\left(1-\frac{\rho_0\alpha_1 \alpha_2(J_1+J_2)}{w_0} \right), \\
 w(J_1,J_2)=& (\alpha_1+  \alpha_2)(\alpha_1 l_1)^{\frac{ \alpha_2}{\alpha_1+ \alpha_2}}
(  \alpha_2 l_2)^{\frac{\alpha_1}{\alpha_1+ \alpha_2}}
-\rho_0\alpha_1 \alpha_2(J_1+J_2),
\end{align*}
where $\rho_0=\int_0^1h^{-1}(x)\,dx.$
Thus, the tangent space to $\omega(N_L)\cdot 1$
restricted on ${\mathcal S}_0\cap \{\tau=1\}$ is
spanned by $(\phi_{J_1},w_{J_1},1, 0)=(\phi_{J_1},-\rho_0\alpha\beta,1, 0)$ and
$(\phi_{J_2},w_{J_2},0, 1)=(\phi_{J_2},-\rho_0\alpha\beta,0, 1)$.
In view of the display in Proposition~\ref{s-u-inter},
the tangent space to
$\alpha(N_R)$ restricted on ${\mathcal S}_0\cap \{\tau=1\}$ is
spanned by $(0,0,1,0)$ and $(0,0,0,1)$. Note that
${\mathcal S}_0\cap \{\tau=1\}$ is four dimensional. Thus, it suffices to show
that the above four vectors are linearly independent, or equivalently,
 $\phi_{J_1}\neq \phi_{J_2}$. The latter can be
verified by a direct computation. Indeed, if $J_1+J_2\neq 0$
at the intersection points, then
\[\phi_{J_1}-\phi_{J_2}
  =\frac{\alpha_1+ \alpha_2}{\alpha_1 \alpha_2 (J_1+J_2)}
   \ln\left(1-\frac{\rho_0\alpha_1 \alpha_2 (J_1+J_2)}{w_0}\right)\neq 0;\]
if $J_1+J_2=0$ at intersection points,  then
$\phi(J_1,J_2)=\nu_0+\rho_0(  \alpha_2 J_2-\alpha_1 J_1)/{w_0}$ and  hence
$\phi_{J_1}-\phi_{J_2}=-\rho_0(\alpha_1+ \alpha_2)/{w_0}\neq 0.$
\end{proof}

\subsection{ \bf Proof of Theorem~\ref{main}}\label{valid}
We provide a detailed version of Theorem~\ref{main} and its proof.   \begin{thm}\label{main2} Assume that $\alpha_1 l_1\neq   \alpha_2 l_2$
and $\alpha_1 r_1\neq   \alpha_2 r_2$.
For $\mu>0$ small, the connecting problem~\eqref{slow}
and~\eqref{left-right} has a unique solution near a singular orbit. The
singular orbit is the union of  two fast orbits of system~\eqref{limfast} and
one slow orbit of system~\eqref{limnewslow}; more precisely, with both
 $J_1$ and $J_2$ given in Theorem~\ref{main},

(i) the fast orbit representing the limiting boundary layer
at $x=0$ lies on $B_L\cap W^s({\mathcal Z}_0)$ from $B_L$ to
$\omega(N_L)\subset {\mathcal Z}_0$ whose starting point has the $u$-component
given by~\eqref{s-U} in Propositions~\ref{s-u-inter},

(ii) the fast orbit representing the limiting boundary layer
at $x=1$ lies on $B_R\cap W^u({\mathcal Z}_0)$ from $B_R$ to
$\alpha(N_R)\subset {\mathcal Z}_0$ whose starting point has the $u$-component
given by~\eqref{u-U} in Propositions~\ref{s-u-inter},

(iii) the slow orbit on ${\mathcal S}_0$
connecting the two boundary layers from $x=0$ to $x=1$ is displayed
in~\eqref{genslow} together with the quantities in Proposition~\ref{s-orbit}.
\end{thm}
\begin{proof} The singular orbit has been studied in Sections~\ref{inner}
and~\ref{outer}, which is summarized in (i), (ii) and (iii) of this theorem.
 It remains to show the existence and uniqueness of a solution
near the singular orbit for $\mu>0$.
Recall that $M_L^{\mu}$ (resp., $M_R^{\mu}$)
is the union of all forward (resp., backward) orbits starting from $B_L$ (resp.,
$B_R$).  It suffices to show that, for $\mu>0$ small,
$M_L^{\mu}$ and $M_R^{\mu}$ intersect transversally with each other
around the singular orbit. We note that the assumption
$\alpha_1 l_1\neq   \alpha_2 l_2$ and $\alpha_1 r_1\neq   \alpha_2 r_2$ imply the vector field
of~\eqref{slow} is not tangent to $B_L$ and $B_R$, and hence,
$M_L^{\mu}$ and $M_R^{\mu}$ are smooth invariant manifolds.

For $\mu>0$ small, the evolutions of
$M_L^{\mu}$ and $M_R^{\mu}$ from $B_L$ and $B_R$, respectively, to an
$\mu$-neighborhood of ${\mathcal Z}_0$ along the two boundary layers
are governed by system~\eqref{fast}. Since, for system~\eqref{limfast},
$M_L^0$ and $M_R^0$ intersect $W^s({\mathcal Z}_0)$ and $W^u({\mathcal Z}_0)$
transversally, we have that  $M_L^{\mu}$  and $M_R^{\mu}$
intersect $W^s({\mathcal Z}_0)$ and $W^u({\mathcal Z}_0)$ transversally.
As discussed in  Section~\ref{outer}, in terms of the blow-up coordinates,
$M_L^{\mu}$  and $M_R^{\mu}$  intersect $W^s({\mathcal S}_0)$ and
$W^u({\mathcal S}_0)$ transversally for system~\eqref{newfast}.
And, if we denote $N_L=M_L^{0}\cap W^s({\mathcal S}_0)$ and
$N_R=M_R^{0}\cap W^u({\mathcal S}_0)$, then the vector field on ${\mathcal S}_0$
is not tangent to   $\omega(N_L)$  and $\alpha(N_R)$.  Furthermore,
the traces $\bar{M}_L$  and $\bar{M}_R$ of $\omega(N_L)$ and $\alpha(N_R)$
respectively under the slow flow on ${\mathcal S}_0$ intersect transversally.
All conditions for the Exchange Lemma (see~\cite{TKJ} and
also~\cite{JK, Jon, JKK})
are satisfied, and hence,  $M^{\mu}_L$ and $M^{\mu}_R$ intersect
transversally. The intersection  has dimension
\[\dim M^{\mu}_L+\dim M^{\mu}_R-7=4+4-7=1,\]
which is the orbit of the unique solution for the connecting problem near
the singular orbit.
\end{proof}

\begin{rem}\label{nolayer} We have considered the situation that
$\alpha_1 l_1\neq   \alpha_2 l_2$ and $\alpha_1 r_1\neq   \alpha_2 r_2$.
In case that $\alpha_1 l_1=  \alpha_2 l_2$ or $\alpha_1 r_1=  \alpha_2 r_2$,
then $B_L$ or $B_R$ are on the slow manifold ${\mathcal S}_0$
and hence there is no boundary layer at $x=0$ or $x=1$.
\end{rem}

 \subsection{A Special Case} We  conclude the paper by examining a special case.
Consider  the one-dimensional limit PNP system~\eqref{pnp}
with the special  boundary conditions
\begin{align}\label{sbv}
\phi(0)  = \phi_0,    \;  \phi (1)=0, \;
\alpha_1 c_1 (0) = \alpha_1 c_1 (1)=  \alpha_2 c_2(0)=  \alpha_2 c_2(1)=k>0.
\end{align}
One sees that $(c_1^0(x),c_2^0(x),\phi^0(x))=(k/{\alpha_1},  k/{\alpha_2}, (1-x)\phi_0)$ is a steady-state solution. Motivated by many works on Lyapunov functions for systems of PNP-types (see, e.g., \cite{BD}), we set
\[L(t)=\sum_{j=1}^2\frac{1}{D_j}\int_0^1h(x)\left(c_j(x,t)-c_j^0(x)\right)\ln\frac{c_j(x,t)}{c_j^0(x)}.\]
 It turns out that $L(t)$ is a Lyapunov functional. In fact, using the equation and integration by parts,
\begin{align*}
L'(t)=& \sum_{j=1}^2\frac{1}{D_j}\int_0^1h(x)\partial_tc_j(x,t)\left( \ln\frac{c_j(x,t)}{c_j^0(x)}+
 \frac{(c_j(x,t)-c_j^0(x))}{c_j(x,t)}\right)\\
=&-\int_0^1h\left(\frac{c_1+c_1^0}{c_1^2}(\partial_xc_1)^2
+\frac{c_2+c_2^0}{c_2^2}(\partial_xc_2)^2\right)\\
&-\lambda\int_0^1h(\alpha_1 c_1-  \alpha_2 c_2)^2-\lambda k\int_0^1h(\alpha_1 c_1-  \alpha_2 c_2)(\ln (\alpha_1 c_1)-\ln(  \alpha_2 c_2))\le 0.
\end{align*}
Also, $L'(t)=0$ if and only if $\partial_xc_1=\partial_xc_2=\alpha_1 c_1-  \alpha_2 c_2=0$, which in turn imply that $c_1=c_1^0$, $c_2=c_2^0$ and $\phi=\phi^0$.  Due to the invariant principle (Proposition \ref{invariantregion} for one-dimensional case), one can check that $L(t)$ is equivalent to
\[\sum_{i=1}^2\int_0^1(c_i(x,t)-c_i^0(x))^2\,dx\]
if $c_i(x,0)>0$ for $x\in [0,1]$, and hence, $c_i\to c_i^0$ in $L^2(0,1)$ exponentially as $t\to\infty$. This   shows a significant difference in asymptotic behavior between the total non-flux boundary conditions (see~\cite{BD}) and the boundary conditions~\eqref{BV}  considered in this work  for the PNP systems.

\bibliographystyle{plain}

\end{document}